\documentclass[leqno, 10pt, a4paper]{amsart}
\usepackage{amsmath}
\usepackage{amssymb, latexsym, mathrsfs, a4wide, mathabx, dsfont}
\usepackage{xcolor}

 \newtheorem{theorem}{Theorem}[section]
 \newtheorem{lemma}[theorem]{Lemma}
 \newtheorem{proposition}[theorem]{Proposition}
 \newtheorem{corollary}[theorem]{Corollary}

 \newtheorem*{theorem*}{Theorem}
\newtheorem*{proposition*}{Proposition}
\newtheorem*{lemma*}{Lemma}

\theoremstyle{definition}
 \newtheorem{definition}{Definition}[section]

 \theoremstyle{remark}
 \newtheorem{example}[theorem]{Example}
 \newtheorem{remark}[theorem]{Remark}

   \newtheorem*{claim*}{Claim}

\newcommand{\op}[1]{\operatorname{#1}}

\newcommand{\norm}[1]{\ensuremath{\|{#1}\|}}

\newcommand{\scal}[2]{\ensuremath{\left\langle #1 | #2 \right\rangle}}

\newcommand{\acoupt}[2]{\ensuremath{(#1,#2)}}

\newcommand{\Tr}{\ensuremath{\op{Tr}}}

\newcommand{\Res}{\ensuremath{\op{Res}}}

\def\Xint#1{\mathchoice
{\XXint\displaystyle\textstyle{#1}}%
{\XXint\textstyle\scriptstyle{#1}}%
{\XXint\scriptstyle\scriptscriptstyle{#1}}%
{\XXint\scriptscriptstyle\scriptscriptstyle{#1}}%
\!\int}
\def\XXint#1#2#3{{\setbox0=\hbox{$#1{#2#3}{\int}$}
\vcenter{\hbox{$#2#3$}}\kern-.5\wd0}}

\def\dashint{\Xint-}

\newcommand{\bint}{\ensuremath{\dashint}}

\newcommand{\GL}{\op{GL}}

\newcommand{\C}{\ensuremath{\mathbb{C}}} 
 
\newcommand{\N}{\ensuremath{\mathbb{N}}} 
 
\newcommand{\R}{\ensuremath{\mathbb{R}}} 
\newcommand{\bS}{\ensuremath{\mathbb{S}}} 
\newcommand{\T}{\ensuremath{\mathbb{T}}} 
\newcommand{\Z}{\ensuremath{\mathbb{Z}}}

\newcommand{\Rn}{\ensuremath{\R^{n}}}

\newcommand{\fS}{\ensuremath{\mathfrak{S}}}

\newcommand{\Ca}[1]{\ensuremath{\mathscr{#1}}}
\newcommand{\cA}{\Ca{A}}

\newcommand{\cH}{\ensuremath{\mathscr{H}}}

\newcommand{\cL}{\ensuremath{\mathscr{L}}}

\newcommand{\cR}{\Ca{R}}
\newcommand{\cS}{\ensuremath{\mathscr{S}}}

\newcommand{\scA}{\mathscr{A}}
\newcommand{\scI}{\mathscr{I}}
\newcommand{\scL}{\mathscr{L}}
\newcommand{\sH}{\mathscr{H}}
\newcommand{\sK}{\mathscr{K}}
\newcommand{\sQ}{\mathscr{Q}}
\newcommand{\sX}{\mathscr{X}}

\newcommand{\psido}{$\Psi$DO} 
\newcommand{\psidos}{$\Psi$DOs}

\newcommand{\stS}{\mathbb{S}}

\newcommand{\Sp}{\op{Sp}}
\newcommand{\Vol}{\op{Vol}}

\newcommand{\Com}{\op{Com}}

\newcommand{\bcn}{\hat{c}_{n}}
\newcommand{\cDelta}{\widehat{\Delta}} 

\numberwithin{equation}{section}

\begin{document}

\title{Connes' trace theorem for Curved Noncommutative Tori. Application to Scalar Curvature}

\author{Rapha\"el Ponge}
 \address{School of Mathematics, Sichuan University, Chengdu, China}
 \email{ponge.math@icloud.com}

 \thanks{The research for this article was partially supported by 
  NSFC Grant 11971328 (China).}
 
\begin{abstract}
In this paper we prove a version of Connes' trace theorem for noncommutative tori of any dimension~$n\geq 2$. This allows us to recover and improve earlier versions of this result in dimension $n=2$ and $n=4$ by Fathizadeh-Khalkhali~\cite{FK:LMP13, FK:JNCG15}. We also recover the 
ConnesÕ integration formula for flat noncommutative tori of McDonald-Sukochev-Zanin~\cite{MSZ:MA19}. As a further application we prove a curved version of this integration formula in terms of the Laplace-Beltrami operator defined by an arbitrary Riemannian metric. For the class of so-called self-compatible Riemannian metrics (including the conformally flat metrics of Connes-Tretkoff) this shows that Connes' noncommutative integral allows us to recover the Riemannian density. 
This exhibits a neat link between this notion of noncommutative integral and noncommutative measure theory in the sense of operator algebras. As an application of these results, we setup a natural notion of scalar curvature for curved noncommutative tori. 
\end{abstract}

\maketitle 

\section{Introduction}
The main goal of noncommutative geometry is to translate the classical tools of differential geometry in the Hilbert space formalism of quantum mechanics~\cite{Co:NCG}. In this framework the role of the integral is played by positive (normalized) traces on the weak trace class $\scL^{1,\infty}$. Important examples of such traces are provided by Dixmier traces~\cite{Di:CRAS66}. An operator $T\in \scL^{1,\infty}$ is measurable (resp., strongly measurable) when the value $\varphi(T)$ is independent of the trace $\varphi$ when it ranges throughout Dixmier traces (resp., positive normalized traces). The NC integral $\bint T$ is then defined as the single value $\varphi(T)$. 

Given any closed Riemannian manifold $(M^n,g)$, Connes' trace theorem~\cite{Co:CMP88, KLPS:AIM13} asserts that every pseudodifferential operator $P$ of order~$-n$ on $M$ is strongly measurable, and we have
\begin{equation}
 \bint P= \frac{1}{n}\Res (P),
 \label{eq:Intro.Trace-thm}
\end{equation}
where $\Res$ is the noncommutative residue trace of Guillemin~\cite{Gu:AIM85} and Wodzicki~\cite{Wo:HDR,Wo:NCR}. Let $\Delta_g$ be the Laplace-Beltrami operator on functions. For operators of the form $f\Delta_g^{-\frac{n}{2}}$, $f\in C^\infty(M)$, we further have ConnesÕ integration formula~\cite{Co:CMP88}, 
\begin{equation}
 \bint f \Delta_g^{-\frac{n}{2}}= c_n \int_M f \nu(g), \qquad c_n:=\frac{1}{n}(2\pi)^{-n}|\bS^{n-1}|,
 \label{eq:Intro.Integration-Formula}
\end{equation}
where $\nu(g)$ is the Riemannian density. In particular, this shows that the noncommutative integral recaptures the Riemannian density. 

This paper deals with analogues for noncommutative tori of the trace theorem~(\ref{eq:Intro.Trace-thm}) and the integration formula~(\ref{eq:Intro.Integration-Formula}). Noncommutative 2-tori naturally arise from actions of $\Z$ on the circle $\bS^1$ by irrational rotations. More generally, an $n$-dimensional noncommutative torus $A_\theta=C(\T_n^\theta)$ is a $C^*$-algebra generated by unitaries $U_1, \ldots, U_n$ subject to  relations, 
\begin{equation*}
 U_lU_j = e^{2i\pi \theta_{jl}}U_jU_l, \qquad j,l=1,\ldots, n, 
\end{equation*}
where $\theta=(\theta_{jl})$ is some given real anti-symmetric matrix. There is a natural action of $\R^n$ onto $A_\theta$. The subalgabra of smooth vectors $\scA_\theta=C^\infty(\T^n_\theta)$ is the smooth noncommutator tori. The Hilbert space $\cH_\theta=L^2(\T^n_\theta)$ is obtained as the space of the GNS representation for the standard normalized trace $\tau$ of $A_\theta$. 

The pseudodifferential calculus on NC tori of Connes~\cite{Co:CRAS80}  has been receiving a great deal of attention recently due to its role in the early development of modular geometry on NC tori by Connes-Tretkoff~\cite{CT:Baltimore11} and Connes-Moscovici~\cite{CM:JAMS14} (see also~\cite{Co:Survey19, FK:Survey19} for surveys). There is a noncommutative residue trace for integer order \psidos\ on $\T_\theta^n$ (see~\cite{FW:JPDOA11, FK:JNCG15, LNP:TAMS16, Po:TracesNCT}). This is even the unique trace up to constant multliple~\cite{Po:TracesNCT}. 

The trace theorem of this paper (Theorem~\ref{thm:Trace-Thm}) states that any pseudodifferential operator $P$ of order $-n$ on $\T_\theta^n$ is strongly measurable, and we have 
\begin{equation}
  \bint P= \frac{1}{n}\Res (P),
  \label{eq:Intro.trace-formula-NCT} 
\end{equation}
where on the r.h.s.~we have the noncommutative residue trace for NC tori. This refines earlier results of Fathizadeh-Khalkhali~\cite{FK:LMP13, FK:JNCG15} in dimensions $n=2$ and $n=4$, where measurability and a trace formula are established, but strong measurability is not. The trace formula~(\ref{eq:Intro.trace-formula-NCT}) can be extended to all operators $aP$ with $a\in A_\theta$ and $P$ as above (see Corollary~\ref{cor:Trace-thm.super-integration-formula}). Applying this to $P=\Delta^{-\frac{n}{2}}$, where $\Delta$ is the flat Laplacian of $\T_\theta^n$, immediately gives back the Connes integration formula for flat  NC tori of McDonald-Sukochev-Zanin~\cite{MSZ:MA19} (see Corollary~\ref{cor:Trace-thm.flat-integration-formula}). 

We seek for a \emph{curved} version of the flat integration formula of~\cite{MSZ:MA19}, i.e., an extension involving the Laplace-Beltrami operator $\Delta_g$ associated with an arbitrary Riemannian metric $g$ on $\T_\theta^n$ which was recently constructed in~\cite{HP:Laplacian}.  We refer to Section~\ref{sec:Riemannian} for the precise definition of Riemannian metrics on $\T_\theta^n$ in the sense of~\cite{HP:Laplacian, Ro:SIGMA13}. They are given by symmetric positive invertible matrices $g=(g_{ij})$ with entries in $\cA_\theta$. Any Riemannian metric $g$ defines Riemannian density $\nu(g)$, which is a positive invertible element of $\cA_\theta$. The original setting of~\cite{CM:JAMS14, CT:Baltimore11} corresponds to the special case of conformally flat metrics $g_{ij}=k^2\delta_{ij}$, $k\in \cA_\theta$, $k>0$. 

The Riemannian density $\nu(g)$ defines a weight $\varphi_g(a):=(2\pi)^n\tau[a\nu(g)]$, $a\in A_\theta$. The very datum of this weight produces a non-trivial modular geometry, even when $g$ is conformally flat. In particular, the GNS construction for $\varphi_g$ produces  two non-isometric $*$-representations of $A_\theta$ and its opposite algebra $A_\theta^\circ$. The corresponding Tomita involution and modular operator $J_g(a):=\nu(g)^{-\frac12} a^*\nu(g)^{\frac12}$ and 
 $\mathbf{\Delta}(a):=\nu(g)^{-1}a \nu(g)$, $a \in A_\theta$. Thus, if we let $\cH_g^\circ$ be the Hilbert space of the GNS representation of $A_\theta^\circ$, then have a $*$-representation $a \rightarrow a^\circ$ of $A_\theta$ in $\cH_\theta^\circ$, where $a^\circ =J_ga^*J_g=\nu(g)^{-\frac12}a \nu(g)^{\frac12}$ acts by left multiplication. 

The Laplace-Beltrami operator $\Delta_g:\cA_\theta \rightarrow \cA_\theta$ of~\cite{HP:Laplacian} is an elliptic 2nd order differential operator with principal symbol  $\nu(g)^{-\frac12}|\xi|_g^2 \nu(g)^{\frac12}$, where $|\xi|_g:=(\sum g^{ij}\xi_i\xi_j)^{\frac12}$ and $g^{-1}=(g^{ij})$ is the inverse matrix of $g$. This is also an essentially selfadjoint operator on $\cH^\circ_g$ with non-negative spectrum. Our curved integration formula  (Theorem~\ref{thm:Curved.integration-formula}) then states that, for every $a\in A_\theta$, the operator $a^\circ \Delta_g^{-\frac{n}{2}}$ is strongly measurable, and we have
\begin{equation}
 \bint a^\circ \Delta_g^{-\frac{n}{2}}= \bcn \tau\big[a \tilde{\nu}(g)\big], \qquad \bcn:=\frac{1}{n}|\bS^{n-1}|,
 \label{eq:Intro.curved-integration-NCTori}  
\end{equation}
where $\tilde{\nu}(g):=|\bS^{n-1}|^{-1} \int_{\bS^{n-1}}|\xi|_g^{-n} d^{n-1}\xi $ is the so-called \emph{spectral Riemannian density} (see Section~\ref{sec:Curved-Integration}). 

The spectral Riemannian density $\tilde{\nu}(g)$ enjoys many of the properties of the Riemannian density $\nu(g)$ (see Proposition~\ref{prop:Curved.properties-tnug}). In general, $\nu(g)$ and $\tilde{\nu}(g)$ need not agree, but they do agree when $g$ is \emph{self-compatible}, i.e., the entries $g_{ij}$ pairwise commute with each other (\emph{loc.\ cit.}). For instance, conformally flat metrics and the functional metrics of~\cite{GK:arXiv18} are self-compatible. 

In the self-compatible case, the integration formula~(\ref{eq:Intro.curved-integration-NCTori}) shows that the NC integral recaptures the  Riemannian weight $\varphi_g$ (see Corollary~\ref{cor:Curved.integration-formula-selfcompatible}). This provides us with a neat link between the NC integral in the sense of noncommutative geometry, and the noncommutative measure theory in the sense of operator algebras, where the role of Radon measures is played by weights on $C^*$-algebras. 

In the same way as with ordinary (compact) manifolds the trace formula~(\ref{eq:Intro.trace-formula-NCT}) allows us to extend the NC integral to \emph{all} \psidos\ on $\T^n_\theta$ even those that are not weak trace class or even bounded (see Definition~\ref{def:curvature.bint}). Together with the curved integration formula this allows us to setup a \emph{quantum length element} $ds:= (c_n^{-1}\Delta_g^{-\frac{n}2})^{\frac1{n}}=c_n^{-\frac1{n}}\Delta_g^{-\frac12}$ and \emph{spectral $k$-dimensional volumes} $\widetilde{\Vol}^{(k)}_g(\T^n_\theta):=\bint ds^k$ for $k=1,\ldots,n$ (see Section~\ref{sec:curvature}). 

For any closed Riemannian manifold $(M^n,g)$, $n\geq 3$, we have 
\begin{equation}
 \bint f ds^{n-2} =-c_n'\int_M \kappa(g) \nu(g) \qquad \forall f \in \C^\infty(M),
 \label{eq:Intro.bint-curvature} 
\end{equation}
where $\kappa(g)$ is the scalar curvature and $c_n'$ is some (positive) universal constant. Given any curved NC tori $(\T^n_\theta, g)$, $n\geq 3$,
 the scalar curvature is the unique $\nu(g)\in \cA_\theta$ such that 
\begin{equation*}
 -\frac{1}{c_n'}\bint a ds^{n-2}= \tau\big[a \kappa(g)\big] \qquad \forall a\in \cA_\theta.  
\end{equation*}
This is consistent with previous approaches to the scalar curvature for NC tori. When $\theta =0$ we recover the usual notion of scalar curvature. 

The scalar curvature $\kappa(g)$ is naturally expressed in terms of some integral of the symbol of degree~$-n$ of $\cDelta_g^{-\frac{n}2+1}$, where $\cDelta_g=\nu(g)^{\frac12}\Delta \nu(g)^{-\frac12}$ is the Laplace-Beltrami operator $\Delta_g$ under the unitary isomorphism $(2\pi)^{\frac{n}2}\nu(g)^{\frac12}:\cH_g^\circ \rightarrow \cH_\theta$ (see Proposition~\ref{prop:curved.kappag}). To relate our description to the scalar curvature to previous descriptions we use the results of~\cite{LP:Resolvent} on complex powers of positive elliptic \psidos\ on NC tori. Namely, we re-express $\kappa(g)$ in terms of the 2nd sub-leading symbol $\sigma_{-4}(\xi;\lambda)$ of the resolvent $(\cDelta_g-\lambda)^{-1}$ (Proposition~\ref{prop:curved.curvature-sigma4}). This provides us with a solid notion of scalar curvature for curved NC tori of dimension~$\geq 3$. In addition, we recover the modular scalar curvature of NC 2-tori of~\cite{CM:JAMS14, CT:Baltimore11} by analytic continuation of the dimension (see Remark~\ref{sec:curvature.2D}).

%sectioning 
This paper is organized as follows. In Section~\ref{sec:Quantized}, we review the main facts on Connes' quantized calculus and the noncommutative integral. 
In Section~\ref{sec:NCtori}, we review the main definitions and properties of noncommutative tori. In Section~\ref{sec:PsiDOs}, we further review the main definitions and properties of pseudodifferential operators on noncommutative tori. 
In Section~\ref{sec:NCR}, we survey the construction of the noncommutative residue trace for \psidos\ on NC tori. 
ln Section~\ref{sec:deltaj-derivatives}, we establish that any \psido\ agrees with a scalar-symbol \psido\ modulo the closure of the weak trace class commutator subspace. 
In Section~\ref{sec:Weyl}, we establish a WeylÕs law for positive \psidos\ with scalar symbols. 
In Section~\ref{sec:Connes-Trace}, we prove the version for noncommutative tori of Connes' trace theorem and give a couple of applications. 
In Section~\ref{sec:Riemannian}, we review the main definitions and properties of Riemannian metrics and Laplace-Beltrami operators.
In Section~\ref{sec:Curved-Integration}, we prove the curved integration formula and present a few consequences. 
In Section~\ref{sec:curvature},  we set up  a natural notion of scalar curvature for curved NC tori. 
In Appendix~\ref{sec:quantized-EH} we sketch a proof of~(\ref{eq:Intro.bint-curvature}). In Appendix~\ref{sec:L2-formulas} we compare 
our integration formulas with the curved integration formula of~\cite{MPSZ:Preprint}.

\subsection*{Acknowlegements} My warmest thanks go Edward McDonald, Fedor Sukochev, and Dmitriy Zanin for stimulating and passionate discussions on Connes' integration formula. I thank Galina Levitina and Alex Usachev for providing a reference for Proposition~\ref{prop:Quantized.continuity-positive-trace}. I also thank University of New South Wales (Sydney, Australia) and University of Qu\'ebec (Montr\'eal, Canada) for their hospitality during the preparation of the paper. 

%\clearpage 
\section{Quantized Calculus and Noncommutative Integral}\label{sec:Quantized}  
In this section, we review the main facts on Connes' quantized calculus and the construction of the noncommutative integral. 

\subsection{Quantized calculus} 
The main goal of the quantized calculus of Connes~\cite{Co:NCG} is to translate into a the Hilbert space formalism of quantum mechanics the main tools of the infinitesimal and integral calculi. In what follows, we let $\sH$ be a (separable) Hilbert space. We mention the first few lines of Connes' dictionary. 

\renewcommand{\arraystretch}{1.2}

\begin{center}
    \begin{tabular}{c|c}  
        Classical & Quantum \\ \hline       
       Complex variable & Operator on $\cH $  \\
      Real variable &  Selfadjoint operator on $\cH $  \\  
 Infinitesimal variable & Compact operator on $\cH $ \\
    %    & & \\ 
       Infinitesimal of order $\alpha>0$  & Compact operator $T$ such that\\ 
                      &  $\mu_{k}(T)=\op{O}(k^{-\alpha})$\\
%                 Integral & Dixmier Trace $\bint$ 
    \end{tabular}
\end{center}

The first two lines arise from quantum mechanics. Intuitively speaking, an infinitesimal is meant to be smaller than any real number. For a bounded operator the condition $\|T\|<\epsilon$ for all $\epsilon>0$ gives $T=0$. This condition can be relaxed into 
\begin{center}
 $\forall \epsilon >0$ $\exists E\subset \sH$ such that $\dim E<\infty$ and $\|T_{|E^\perp}\|<\epsilon$.  
\end{center}
This condition is equivalent to $T$ being a compact.  

Denote by $\sK$ the closed two-sided ideal of compact operator. The order of compactness of an operator $T\in \sK$ is given by its characteristic values (a.k.a.~singular values), 
\begin{align*}
 \mu_k(T) & = \inf\left\{\|T_{|E^\perp}\|; \ \dim E=k\right\}\\
 & = \text{$(k+1)$-th eigenvalue of $|T|=\sqrt{T^*T}$}.
\end{align*}
The last line is the min-max principle. Each eigenvalue is counted according to multiplicity. We record the main properties of the characteristic values (see, e.g., \cite{GK:AMS69, Si:AMS05}), 
\begin{gather}
 \mu_k(T)=\mu_k(T^*)=\mu_k(|T|),\qquad 
 \mu_{k+l}(S+T)\leq \mu_k(S) + \mu_l(T),\
 \label{eq:Quantized.properties-mun1} \\
 \mu_k(ATB)\leq \|A\| \mu_k(T) \|B\|,\qquad 
 \mu_k(U^*TU)= \mu_k(T) ,
  \label{eq:Quantized.properties-mun2}
\end{gather}
where $S,T\in \sK$, $A,B\in \scL(\sH)$ and $U\in\scL(\sH)$ is any unitary operator. 

An \emph{infinitesimal operator} of order $\alpha>0$ is any compact operator such that $\mu_k(T)=\op{O}(k^{-\alpha})$. For $p>0$ the weak Schatten class $\scL^{p,\infty}$ is defined by
\begin{equation*}
 \scL^{p,\infty}:=\left\{T\in \sK; \ \mu_k(T)=\op{O}(k^{-\frac1p})\right\}. 
\end{equation*}
Thus, $T$ is an infinitesimal operator of order $\alpha$ if and only if $T\in \scL^{p,\infty}$ with $p=\alpha^{-1}$. 

We refer to~\cite{LSZ:Traces}, and the references therein, for background on the weak Schatten classes $\scL^{p,\infty}$. In particular, they are quasi-Banach ideals (see, e.g., \cite{Si:TAMS76}) with quasi-norms,
\begin{equation*}
 \|T\|_{p,\infty} = \sup_{k\geq 0} (k+1)^{\frac1p} \mu_k(T), \qquad T\in \scL^{p,\infty}.  
 \end{equation*}
For $p>1$ this quasi-norm is equivalent to a norm, and so $\cL^{p,\infty}$ is actually a Banach ideal. In general (see, e.g., \cite{Si:TAMS76}), we have 
\begin{equation}
 \|S+T\|_{p,\infty}\leq 2^{\frac1p}\left(\|S\|_{p,\infty} + \|T\|_{p,\infty}\right), \qquad S,T\in \scL^{p,\infty}.
 \label{eq:Quantized.quasi-norm} 
\end{equation}

For $p=1$ the weak trace class $\cL^{1,\infty}$ is strictly contained in the Dixmier-Macaev ideal,
 \begin{equation*}
  \scL^{(1,\infty)}:=\bigg\{T\in \sK; \ \sum_{k<N}\mu_k(T)=\op{O}(\log N)\bigg\}. 
\end{equation*}
The Dixmier-Macaev ideal is a Banach ideal. It is also called the weak trace class by various authors. We follow the convention of~\cite{LSZ:Traces} where, under Calkin's correspondance, $\cL^{p,\infty}$ is the image of the weak-$\ell^p$ sequence space $\ell^{p,\infty}=\big\{(a_k)\in \C^{\N_0}; |a_k|=\op{O}(k^{-\frac1p})\big\}$ .   

We also record the inclusions, 
\begin{equation*}
 \scL^p \subset \scL^{p,\infty} \subset \scL^1 \subset \scL^{1,\infty}\subset \scL^{(1,\infty)}, \qquad p>1.
\end{equation*}
All these inclusions are strict. 

\subsection{Noncommutative integral} We seek for an analogue of the notion of integral in the setting of the quantized calculus. As an Ansatz this should be a linear functional satisfying the following conditions: 
\begin{enumerate}
 \item[(I1)] It is defined on all infinitesimal operators of order~1. 
 
 \item[(I2)] It vanishes on infinitesimal operators of order~$> 1$. 
 
 \item[(I3)] It takes non-negative values on positive operators. 
 
 \item[(I4)] It is invariant under Hilbert space isomorphisms. 
 \end{enumerate}
The conditions (I1) and (I4) lead us to seek for a trace on $\scL^{1,\infty}$, where by a trace we mean any linear form that is annihilated by the commutator subspace, 
\begin{equation*}
 \Com\big(\scL^{1,\infty}\big):=\op{Span}\left\{[A,T];\ A\in \scL(\sH), \ T\in\scL^{1,\infty}\right\}. 
\end{equation*}
It can be shown that any trace on $\scL^{1,\infty}$ vanishes on trace-class operators (see Proposition~\ref{prop:Quantized.trace-class-commutators} below). As every infinitesimal operators of order~$>1$ is trace-class, the condition (I2) is always satisfied by traces on $\scL^{1,\infty}$ (see Proposition~\ref{prop:Quantized.trace-class-commutators} below). The condition (I3) further requires the NC integral to be positive. Therefore, we are seeking for a positive trace on the quasi-Banach space $\scL^{1,\infty}$. 

There is a whole zoo of traces on $\scL^{1,\infty}$. We refer to~\cite{LSZ:Traces, LSZ:Survey19}, and the references therein, for a detailed account on traces on $\scL^{1,\infty}$ and the commutator subspace $ \Com(\scL^{1,\infty})$. We will only need the following two results. 

\begin{proposition}[\cite{DFKW:AIM04, LSZ:Traces}]\label{prop:Quantized.trace-class-commutators}
Every trace-class operator is contained in $\Com(\scL^{1,\infty})$. In particular, every trace on $\scL^{1,\infty}$ is annihilated by $\cL^1$. 
\end{proposition}

\begin{proposition}\label{prop:Quantized.continuity-positive-trace}
 Every positive trace on $\scL^{1,\infty}$ is continuous.
\end{proposition}

\begin{remark}
Any positive linear form on a $C^*$-algebra is continuous (see~\cite[Theorem~3.3.1]{Mu:AP90}). The same type argument shows that on any positive linear form on $\scL^{1,\infty}$ is continuous. 
\end{remark}

\begin{remark}
 Proposition~\ref{prop:Quantized.continuity-positive-trace} has a converse. Namely, every continuous trace on $\scL^{1,\infty}$  is a linear combination of 4 positive traces (see~\cite[Corollary~2.2]{CMSZ:ETDS19}). Therefore, the positive traces on $\scL^{1,\infty}$ span the space of continuous traces. 
\end{remark}

A well-known example of positive trace is the Dixmier trace~\cite{Di:CRAS66}. We briefly present its construction by following~\cite[Appendix~A]{CM:GAFA95}. In what follows we denote by $\cL^{1,\infty}_+$ the  cone of positive operators in $\cL^{1,\infty}$. Given any $T\in \cL^{1,\infty}_+$, for $u>0$ and $\lambda\geq e$, we set
\begin{equation*}
 \sigma_u(T):=\int_0^s \mu_{[s]}(t)ds, \qquad \tau_\lambda(T)= \frac{1}{\log \lambda} \int_e^\lambda \frac{\sigma_u(T)}{\log u} \frac{du}{u}.
\end{equation*}
Note that $\sigma_N(T)=\sum_{k<N}\mu_k(T)$ for every integer $N\geq 1$. The fact that $T\in \cL^{1,\infty}_+$ ensures us that the function $\lambda \rightarrow \tau_\lambda(T)$ is in the $C^*$-algebra $C_b[e,\infty)$ of bounded continuous functions on $[e,\infty)$. Denote by $C_0[e,\infty)$ the (closed) ideal of continuous functions vanishing at $\infty$. It can be shown that, for all $S,T\in \cL^{1,\infty}_+$, we have
\begin{equation}
 \tau_\lambda (S+T)= \tau_\lambda (T) + \tau_\lambda (S) \quad \bmod C_0[e,\infty). 
 \label{eq:Quantized.asymptotic-additivity} 
\end{equation}

Let $\omega$ be a state on the quotient $C^*$-algebra $A:=C_b[e,\infty) \slash C_0[e,\infty)$ (i.e., $\omega$ is a positive linear functional such that $\omega(1)=1$). Then~(\ref{eq:Quantized.asymptotic-additivity}) ensures us that we define an additive functional $\varphi_\omega : \cL^{1,\infty}_+\rightarrow [0,\infty)$ by letting
\begin{equation*}
 \varphi_w(T)= \omega\left(\big[\tau(T)\big]\right), \qquad T \in \cL^{1,\infty}_+,
\end{equation*}
where $[\tau(T)]$ is the class of $\lambda \rightarrow \tau_\lambda(T)$ in $A$. This uniquely extends to a positive linear trace $\varphi_\omega:\cL^{1,\infty}\rightarrow \C$. This trace is called the \emph{Dixmier trace} associated with $\omega$. We note that, for every $T\in \cL^{1,\infty}_+$, we have 
\begin{equation}
 \lim_{N\rightarrow \infty} \frac{1}{\log N} \sum_{k<N} \mu_k(T)= L \Longrightarrow \varphi_\omega(T)=L.
 \label{eq:Quantized.Dixmier-measurable}
\end{equation}
In particular, in this case the value of $\varphi_\omega(T)$ is independent of the choice of $\omega$. 

\begin{definition}[Connes~\cite{Co:NCG}] 
An operator $T\in \scL^{1,\infty}$ is called \emph{measurable} when the value of $\varphi_\omega(T)$ is independent of the choice of the state $\omega$. We then define the \emph{noncommutative integral} $\bint T$ by
\begin{equation*}
 \bint T := \varphi_\omega(T), \qquad \text{$\omega$ any state on $\sQ[e,\infty)$}. 
\end{equation*}
\end{definition}

Let $T_0$ be any operator on $\scL^{1,\infty}_+$ such that $\mu_k(T_0)=(k+1)^{-1}$ for all $k\geq 0$, i.e., there is an orthonormal basis $(\xi_k)_{k\geq 0}$ of $\sH$ such that $T_0 \xi_k=(k+1)^{-1}\xi_k$ for all $k\geq 0$. It follows from~(\ref{eq:Quantized.Dixmier-measurable}) that $\varphi_\omega(T_0)=1$ for every state $\omega$. More generally, we say that a trace $\varphi$ on $\scL^{1,\infty}$ is \emph{normalized} when $\varphi(T_0)=1$. This condition does not depend on the choice of $T_0$. 

There are positive normalized traces on $\scL^{1,\infty}$ that are not Dixmier traces. In fact, there are even positive normalized traces on  $\scL^{1,\infty}$ that do not extend to the Dixmier-Macaev ideal $\scL^{(1,\infty)}$ (see~\cite[Theorem~4.7]{SSUZ:AIM15}). Therefore, it stands for reason to consider a stronger notion of measurability. 

\begin{definition}
An operator $T\in \scL^{1,\infty}$ is called \emph{strongly measurable} when there is $L\in \C$ such that $\varphi(T)=L$ for every normalized positive trace $\varphi$ on $\scL^{1,\infty}$. We then define the \emph{noncommutative integral} $\bint T$ by
\begin{equation*}
 \bint T := \varphi(T),  \qquad \text{$\varphi$ any normalized positive trace on $\scL^{1,\infty}$}. 
\end{equation*}
 \end{definition}

\begin{remark}
 The class of strongly measurable operators is strictly contained in the class of measurable operators. We refer to~\cite[Theorem~7.4]{SSUZ:AIM15} for an example of measurable operator that is not strongly measurable.
\end{remark}

\subsection{Connes' trace theorem} Let $(M^n,g)$ be a closed Riemannian manifold. We refer to~\cite[Chapter~16]{Le:Springer12} for background on smooth densities on manifolds and their integrals. The Riemannian density $\nu(g)$ is given in local coordinates by integration against $\sqrt{g(x)}$. Let $L^2_g(M)$ be Hilbert space of $L^2$-functions on $M$ equipped with the inner product defined by $\nu(g)$, i.e., 
\begin{equation*}
 \scal{u}{v}_g=\int_M u\overline{v} \nu(g), \qquad u,v\in L^2_g(M). 
\end{equation*}

Let $\Psi^m(M)$, $m\in \Z$, be the space of (classical) $m$-th order pseudodifferential operators $P:C^\infty(M)\rightarrow C^\infty(M)$. Any $P\in \Psi^m(M)$ with $m\leq 0$ uniquely extends to a bounded operator $P:L^2_g(M)\rightarrow L^2_g(M)$. When $m<0$, any $P\in \Psi^{m}(M)$ is in the weak Schatten class $\scL^{p,\infty}$ with $p:=n|m|^{-1}$, i.e., $P$ is an infinitesimal operator of order~$\geq |m|n^{-1}$. In particular, we get weak trace-class operators when $m=-n$. 

Set $\Psi^\Z(M)=\bigcup_{m\in \Z}\Psi^m(M)$; this is a subalgebra of $\scL(C^\infty(M))$. The noncommutative residue trace $\Res:\Psi^\Z(M)\rightarrow \C$  of Guillemin~\cite{Gu:AIM85}  and Wodzicki~\cite{Wo:HDR} is defined by 
\begin{equation}
 \Res (P)= \int_M c_P(x), \qquad P\in \Psi^\Z(M), 
 \label{eq:Quantized.NCR1} 
\end{equation}
where $c_P(x)$ is the smooth density on $M$ given in local coordinates by 
\begin{equation*}
 c_P(x)=(2\pi)^{-n} \int_{\bS^{n-1}} p_{-n}(x,\xi) d^{n-1}\xi, 
\end{equation*}
where $p_{-n}(x,\xi)$ is the homogeneous symbol of degree~$-n$ of $P$.

\begin{theorem}[Connes' trace theorem~\cite{Co:CMP88, KLPS:AIM13}] Every operator $P\in \Psi^{-n}(M)$ is strongly measurable, and we have 
\begin{equation}
 \bint P = \frac{1}{n} \Res (P).
 \label{eq:Quantized.trace-formula} 
\end{equation}
\end{theorem}
 
\begin{remark}
 Connes~\cite{Co:CMP88} only considered Dixmier traces, so he only established the measurablity of operators in $P\in \Psi^{-n}(M)$ along with the trace formula~(\ref{eq:Quantized.trace-formula}). The strong measurability is a consequence of the results of~\cite{KLPS:AIM13}, where the trace formula is established 
for \emph{every} trace $\varphi$ on $\scL^{1,\infty}$. In fact, only a slight elaboration of Connes' original proof is needed in order to get strong measurability. 
\end{remark}
 
\begin{remark}
 We refer to~\cite[Corollary 7.23]{KLPS:AIM13} for an example of non-classical pseudodifferential operator that is not measurable. 
\end{remark}

Together with the formula~(\ref{eq:Quantized.NCR1}) the trace formula~(\ref{eq:Quantized.trace-formula}) allows us to compute the NC integral of any $P\in \Psi^{-n}(M)$ from the sole knowledge of its principal symbol in local coordinates. As an application this allows us to recover the Riemannian density $\nu(g)(x)$ as follows. 

Let $\Delta_g:C^\infty(M)\rightarrow C^\infty(M)$ be the (positive) Laplace-Beltrami operator. This is an selfadjoint elliptic 2nd order differential operator with non-negative spectrum. The power $\Delta_g^{-\frac{n}{2}}$ is an operator in $\Psi^{-n}(M)$ with principal symbol $|\xi|_g^{-n}$, where $|\cdot|_g$ is the Riemannian metric on the cotangent bundle $T^*M$ (i.e., $|\xi|_g^2=\sum g^{ij}\xi_i\xi_j$). By applying the trace formula~(\ref{eq:Quantized.trace-formula}) to $f\Delta_g^{-\frac{n}{2}}$, $f\in C^\infty(M)$, we arrive at the following result. 

\begin{theorem}[ConnesÕ integration formula~\cite{Co:CMP88}] For every $f\in C^\infty(M)$, the operator $f\Delta_g^{-\frac{n}2}$ is strongly measurable, and we have
\begin{equation}
 \bint f \Delta_g^{-\frac{n}2} = c_n \int_M f\nu(g), \qquad \text{where}\ c_n:=\frac{1}{n}(2\pi)^{-n}|\bS^{n-1}|.
 \label{eq:Quantized.integration-formula} 
\end{equation}
\end{theorem}

\begin{remark}
Connes' integration formula even holds for any $f(x)\in L^2_g(M)$ (\emph{cf.}~\cite{KLPS:AIM13, LPS:JFA10}). 
\end{remark}

%\clearpage
One upshot of the trace formula~(\ref{eq:Quantized.trace-formula}) is the extension of the NC integral to \emph{all} \psidos, including \psidos\ that need not be weak trace class or even bounded. 
Namely, as the noncommutative residue is defined for all \psidos, it is natural to use the r.h.s.~of~(\ref{eq:Quantized.trace-formula}) as a definition of the NC integral for \psidos. That is, we set
\begin{equation}
 \bint P := \frac{1}{n} \Res (P) \qquad \text{for any $P\in \Psi^\Z(M)$}.
 \label{eq:quantized.extension-bint} 
\end{equation}
 
The integration formula~(\ref{eq:Quantized.integration-formula})  shows that the NC integral recaptures the Riemannian density. Equivalentely, this allows us to interpret $c_n^{-1} \Delta_g^{-\frac{n}2}$ a \emph{quantum volume element}. Thinking of the length element as the $n$-th root of the volume element, this leads us to interpret the operator $ds:=(c_n^{-1}\Delta_g^{-\frac{n}2})^{\frac1n}=c_n^{-\frac1n}\Delta_g^{-\frac12}$ as a \emph{quantum length element}. It is then natural to interpret  $\bint ds^k$ as a \emph{$k$-dimensional volume} for $k=1,\ldots, n$. Note this uses the extension of the NC integral to all \psidos\ since $ds^k$ is a \psido\ of order~$-k>-n$ when $k<n$. We refer to~\cite{Po:LMP08} for an elaboration of this line of thought in terms of the Dirac operator. 

When $k=n-2$ and $n\geq 3$ we obtain 
\begin{equation}
 \bint f ds^{n-2} = -c_n' \int_M f \kappa(g) \nu(g) \qquad \text{for all $f\in C^\infty(M)$},
 \label{eq:Quantized.scalar-curvature}
\end{equation}
where $\kappa(g)$ is the scalar curvature and $c_n'$ is a universal constant given by
\begin{equation}
 \frac{1}{c_n'}= 3n (4\pi)^{\frac{n}2} c_n^{\frac{n-2}{n}} \Gamma\big(\frac{n}{2}-1\big). 
 \label{eq:quantized.cn'} 
\end{equation}
This result is part of folklore. For reader's convenience a proof is included in Appendix~\ref{sec:quantized-EH}. 

When $f(x)\equiv 1$ the r.h.s.~ of~(\ref{eq:Quantized.scalar-curvature}) agrees up to a constant with the Einstein-Hilbert action. This provides us with  quantum interpretations of the Einstein-Hilbert action and of the scalar curvature. These interpretations lie at the very roots of the spectral action principal of Chamseddine-Connes~\cite{CC:CMP97}  and the concept of modular scalar curvature of NC tori  by Connes-Tretkoff~\cite{CT:Baltimore11} and Connes-Moscovici~\cite{CM:JAMS14} (see also~\cite{CS:Survey19, Co:Survey19, FK:Survey19} for recent surveys of those topics).  
 
\section{Noncommutative Tori} \label{sec:NCtori}
In this section, we review the main definitions and properties of noncommutative $n$-tori, $n\geq 2$. We refer to~\cite{Co:NCG, HLP:Part1, Ri:CM90}, and the references therein, for a more comprehensive account.
 
Throughout this paper, we let $\theta =(\theta_{jk})$ be a real anti-symmetric $n\times n$-matrix. Denote by $\theta_1, \ldots, \theta_n$ its column vectors.  We also let  $L^2(\T^n)$ be the Hilbert space of $L^2$-functions on the ordinary torus $\T^n=\R^n\slash 2\pi \Z^n$ equipped with the  inner product, 
\begin{equation} \label{eq:NCtori.innerproduct-L2}
 \scal{\xi}{\eta}= (2\pi)^{-n} \int_{\T^n} \xi(x)\overline{\eta(x)}d x, \qquad \xi, \eta \in L^2(\T^n). 
\end{equation}
 For $j=1,\ldots, n$, let $U_j:L^2(\T^n)\rightarrow L^2(\T^n)$ be the unitary operator defined by 
 \begin{equation*}
 \left( U_j\xi\right)(x)= e^{ix_j} \xi\left( x+\pi \theta_j\right), \qquad \xi \in L^2(\T^n). 
\end{equation*}
 We then have the relations, 
 \begin{equation} \label{eq:NCtori.unitaries-relations}
 U_kU_j = e^{2i\pi \theta_{jk}} U_jU_k, \qquad j,k=1, \ldots, n. 
\end{equation}

The \emph{noncommutative torus} $A_\theta=C(\T^n_\theta)$ is the $C^*$-algebra generated by the unitary operators $U_1, \ldots, U_n$.  For $\theta=0$ we obtain the $C^*$-algebra $C(\T^n)$ of continuous functions on the ordinary $n$-torus $\T^n$. Note that~(\ref{eq:NCtori.unitaries-relations}) implies that $A_\theta$ is the closure in $\cL(L^2(\T^n))$ of the linear span of the unitary operators, 
 \begin{equation*}
 U^k:=U_1^{k_1} \cdots U_n^{k_n}, \qquad k=(k_1,\ldots, k_n)\in \Z^n. 
\end{equation*}

 Let $\tau:\cL(L^2(\T^n))\rightarrow \C$ be the state defined by the constant function $1$, i.e., 
 \begin{equation*}
 \tau (T)= \scal{T1}{1}=\int_{\T^n} (T1)(x) d  x, \qquad T\in \cL\left(L^2(\T^n)\right).
\end{equation*}
This induces a continuous tracial state on the $C^*$-algebra $A_\theta$ such that $\tau(1)=1$ and $\tau(U^k)=0$ for $k\neq 0$. 
 The GNS construction then allows us to associate with $\tau$ a $*$-representation of $A_\theta$ as follows. 
 
 Let $\scal{\cdot}{\cdot}$ be the sesquilinear form on $A_\theta$ defined by
\begin{equation}
 \scal{u}{v} = \tau\left( uv^* \right), \qquad u,v\in A_\theta. 
 \label{eq:NCtori.cAtheta-innerproduct}
\end{equation}
Note that the family $\{ U^k; k \in \Z^n\}$ is orthonormal with respect to this sesquilinear form. We let $\cH_\theta=L^2(\T^n_\theta)$ the Hilbert space arising from the completion of $\cA_\theta^0$ with respect to the pre-inner product~(\ref{eq:NCtori.cAtheta-innerproduct}). The action of $A_\theta$ on itself by left-multiplication uniquely extends to a $*$-representation of $A_\theta$ in $\cH_\theta$. When $\theta=0$ we recover the Hilbert space $L^2(\T^n)$ with the inner product~(\ref{eq:NCtori.innerproduct-L2}) and the representation of $C(\T^n)$ by bounded multipliers. In addition,  as $(U^k)_{k \in \Z^n}$ is an orthonormal basis of $\cH_\theta$, every $u\in \cH_\theta$ can be uniquely written as 
\begin{equation} \label{eq:NCtori.Fourier-series-u}
 u =\sum_{k \in \Z^n} u_k U^k, \qquad u_k:=\scal{u}{U^k}, 
\end{equation}
where the series converges in $\cH_\theta$. When $\theta =0$ we recover the Fourier series decomposition in  $L^2(\T^n)$. 

The natural action of $\R^n$ on $\T^n$ by translation gives rise to an action on $\cL(L^2(\T^n))$. This induces a $*$-action $(s,u)\rightarrow \alpha_s(u)$ on $A_\theta$ given by 
\begin{equation*}
% \label{eq:NCtori.action-on-U^k}
\alpha_s(U^k)= e^{is\cdot k} U^k, \qquad  \text{for all $k\in \Z^n$ and $s\in \R^n$}. 
\end{equation*}
This action is strongly continuous, and so we obtain a $C^*$-dynamical system $(A_\theta, \R^n, \alpha)$. We are especially interested in the subalgebra $\cA_\theta=C^\infty(\T^n_\theta)$ of smooth elements of this $C^*$-dynamical system (a.k.a.~\emph{smooth noncommutative torus}). Namely,  
\begin{equation*}
 \cA_\theta:=\biggl\{ u \in A_\theta; \ \alpha_s(u) \in C^\infty(\R^n; A_\theta)\biggr\}. 
\end{equation*}
The unitaries $U^k$, $k\in \Z^n$, are contained in $\cA_\theta$, and so $\cA_\theta$ is a dense subalgebra of $A_\theta$. Denote by $\cS(\Z^n)$ the space of rapid-decay sequences with complex entries. In terms of the Fourier series decomposition~(\ref{eq:NCtori.Fourier-series-u}) we have
\begin{equation*}
 \cA_\theta=\bigg\{ u=\sum_{k\in \Z^n} u_k U^k; (u_k)_{k\in \Z^n}\in  \cS(\Z^n)\bigg\}. 
\end{equation*}
When $\theta=0$ we recover the algebra $C^\infty(\T^n)$ of smooth functions on the ordinary torus $\T^n$ and the Fourier-series description of this algebra. 

For $j=1,\ldots, n$, let $\delta_j:\cA_\theta\rightarrow \cA_\theta$ be the  derivation defined by 
\begin{equation*}
 \delta_j(u) = D_{s_j} \alpha_s(u)|_{s=0}, \qquad u\in \cA_\theta, 
\end{equation*}
where we have set $D_{s_j}=\frac{1}{i}\partial_{s_j}$. When $\theta=0$ the derivation $\delta_j$ is just the derivation $D_{x_j}=\frac{1}{i}\frac{\partial}{\partial x_j}$ on $C^\infty(\T^n)$. In general, for $j,l=1,\ldots, n$, we have
\begin{equation*}
 \delta_j(U_l) = \left\{ 
 \begin{array}{ll}
 U_j & \text{if $l=j$},\\
 0 & \text{if $l\neq j$}. 
\end{array}\right.
\end{equation*}
More generally, given any multi-order $\beta \in \N_0^n$, define 
\begin{equation*}
 \delta^\beta(u) = D_s^\beta \alpha_s(u)|_{s=0} = \delta_1^{\beta_1} \cdots \delta_n^{\beta_n}(u), \qquad u\in \cA_\theta. 
\end{equation*}
We endow $\cA_\theta$ with the locally convex topology defined by the semi-norms,
\begin{equation}
 \cA_\theta \ni u \longrightarrow \left\|\delta^\beta (u)\right\| ,  \qquad \beta\in \N_0^n. 
\label{eq:NCtori.cAtheta-semi-norms}
\end{equation}
With the involution inherited from $A_\theta$ this turns $\cA_\theta$ into a (unital)  Fr\'echet $*$-algebra. The Fourier series~(\ref{eq:NCtori.Fourier-series-u}) of every $u\in \cA_\theta$ converges in $\cA_\theta$ with respect to this topology. In addition, it can be shown that $\cA_\theta$ is closed under holomorphic functional calculus 
(see, e.g., \cite{Co:AIM81, HLP:Part1}). 

\section{Pseudodifferential Operators on Noncommutative Tori} \label{sec:PsiDOs}
In this section, we review the main definitions and properties of pseudodifferential operators on noncommutative tori. 

\subsection{Symbol classes}
There are various classes of symbols on noncommutative tori. 

\begin{definition}[Standard Symbols; see~\cite{Ba:CRAS88, Co:CRAS80}]
$\stS^m (\Rn ; \cA_\theta)$, $m\in\R$, consists of maps $\rho(\xi)\in C^\infty (\Rn ; \cA_\theta)$ such that, for all multi-orders $\alpha$ and $\beta$, there exists $C_{\alpha \beta} > 0$ such that
\begin{equation*} 
%\label{eq:Symbols.standard-estimates}
\norm{\delta^\alpha \partial_\xi^\beta \rho(\xi)} \leq C_{\alpha \beta} \left( 1 + | \xi | \right)^{m - | \beta |} \qquad \forall \xi \in \R^n .
\end{equation*}
\end{definition}
 
\begin{definition} 
  $\cS(\R^n; \cA_\theta)$ consists of maps $\rho(\xi)\in C^\infty (\Rn ; \cA_\theta)$ such that, for all  $N\geq 0$ and multi-orders $\alpha$, $\beta$, there exists $C_{N\alpha \beta} > 0$ such that
\begin{equation*} 
%\label{eq:Symbols.standard-estimates}
\norm{\delta^\alpha \partial_\xi^\beta \rho(\xi)} \leq C_{N\alpha \beta} \left( 1 + | \xi | \right)^{-N} \qquad \forall \xi \in \R^n .
\end{equation*}
\end{definition}

\begin{remark}\label{rmk:Symbols.symbols-intersection}
$\cS(\R^n; \cA_\theta)=\bigcap_{m\in\R}\stS^m( \R^n;\cA_\theta)$.
\end{remark}

\begin{definition}[Homogeneous Symbols] 
$S_q (\R^n; \cA_\theta )$, $q \in \C$, consists of  maps $\rho(\xi) \in C^\infty(\R^n\backslash 0;\cA_\theta)$ that are homogeneous of degree $q$, i.e., 
$\rho( t \xi ) =t^q \rho(\xi)$ for all $\xi \in \R^n \backslash 0$ and $t > 0$. 
\end{definition}

\begin{remark}
 If $\rho(\xi)\in S_q(\R^n;\cA_\theta)$ and $\chi(\xi)\in C^\infty_c(\R^n)$ is such that $\chi(\xi)=1$ near $\xi=0$, then $(1-\chi(\xi))\rho(\xi)\in \stS^{\Re q}( \R^n;\cA_\theta)$. 
\end{remark}

\begin{definition}[Classical Symbols; see \cite{Ba:CRAS88}]\label{def:Symbols.classicalsymbols}
$S^q (\R^n; \cA_\theta )$, $q \in \C$, consists of maps $\rho(\xi)\in C^\infty(\R^n;\cA_\theta)$ that admit an asymptotic expansion,
\begin{equation*}
\rho(\xi) \sim \sum_{j \geq 0} \rho_{q-j} (\xi),  \qquad \rho_{q-j} \in S_{q-j} (\R^n; \cA_\theta ), 
\end{equation*}
where $\sim$ means that, for all $N\geq 0$ and multi-orders $\alpha$, $\beta$, there exists $C_{N\alpha\beta} >0$ such that, for all $\xi \in \R^n$ with $| \xi | \geq 1$, we have
\begin{equation} \label{eq:Symbols.classical-estimates}
\Big\| \delta^\alpha \partial_\xi^\beta \big( \rho - \sum_{j<N} \rho_{q-j} \big)(\xi)\Big\| \leq C_{N\alpha\beta} | \xi |^{\Re{q}-N-| \beta |} .
\end{equation}
\end{definition}

\begin{remark} \label{rmk:Symbols.classical-inclusion}
$S^q(\R^n;\cA_\theta)\subset \stS^{\Re{q}}(\R^n;\cA_\theta)$. 
\end{remark}

\begin{example}
 Any polynomial map $\rho(\xi)=\sum_{|\alpha|\leq m} a_\alpha \xi^\alpha$, $a_\alpha\in \cA_\theta$, $m\in \N_0$, is in $S^m(\R^n;\cA_\theta)$.  
\end{example}

\begin{remark}
 It is also convenient to consider scalar-valued symbols. In this case we denote by $\stS^m(\R^n)$, $\cS(\R^n)$, $S_q(\R^n)$ and $S^q(\R^n)$ the corresponding  classes of scalar-valued symbols. We will consider them as sub-classes of the $\cA_\theta$-valued symbol classes via the natural embedding of $\C$ into $\cA_\theta$. 
\end{remark}

\subsection{Pseudo-differential operators}
Given $\rho(\xi)\in \stS^m(\R^n;\cA_\theta)$, $m\in \R$. We let $P_\rho:\cA_\theta \rightarrow \cA_\theta$ be the linear operator defined by
\begin{equation} \label{eq:PsiDOs.PsiDO-definition}
P_\rho u = (2\pi)^{-n}\iint e^{is\cdot\xi}\rho(\xi)\alpha_{-s}(u)ds d\xi, \qquad u \in \cA_\theta. 
\end{equation}
The above integral is meant as an oscillating integral (see~\cite{HLP:Part1}). Equivalently, for all $u=\sum_{k\in \Z^n} u_k U^k$ in $\cA_\theta$,  we have
\begin{equation} \label{eq:toroidal.Prhou-equation}
                P_{\rho}u = \sum_{k\in \Z^n} u_k \rho(k)U^k.  
\end{equation}
In any case, this defines a continuous linear operator $P_\rho:\cA_\theta \rightarrow \cA_\theta$ (see~\cite{HLP:Part1}).

\begin{definition}
$\Psi^q(\T^n_\theta)$,  $q\in \C$, consists of all linear operators $P_\rho:\cA_\theta\rightarrow \cA_\theta$ with $\rho(\xi)$ in $S^q(\R^n; \cA_\theta)$.
\end{definition}

\begin{remark} \label{rem:PsiDOs.symbol-uniqueness}
If $P=P_\rho$ with $\rho(\xi)$ in $S^q(\R^n; \cA_\theta)$, $\rho(\xi)\sim \sum \rho_{q-j}(\xi)$, then $\rho(\xi)$ is called a \emph{symbol} for $P$. This symbol is not unique, but its restriction to $\Z^n$ and its class modulo $\cS(\R^n;\cA_\theta)$ are unique (see~\cite{HLP:Part1}). As a result, the homogeneous symbols $\rho_{q-j}(\xi)$ are uniquely determined by $P$. The leading symbol $\rho_q(\xi)$ is called the \emph{principal symbol} of $P$. 
\end{remark}

\begin{example}
 A differential operator on $\cA_\theta$ is of the form $P=\sum_{|\alpha|\leq m}a_\alpha\delta^\alpha$, $a_\alpha\in\cA_\theta$ (see~\cite{Co:CRAS80, Co:NCG}). This is a \psido\ of order $m$ with symbol $\rho(\xi)= \sum a_\alpha \xi^\alpha$ (see~\cite{HLP:Part1}).
\end{example}

Let $\cA_\theta'$ be the topological dual of $\cA_\theta$ equipped with its strong dual topology. Note that $\cA_\theta$ embeds into $\cA_\theta'$ as a dense subspace (see, e.g., \cite{HLP:Part1}). A  linear operator $R:\cA_\theta \rightarrow \cA_\theta'$ is called \emph{smoothing} when its range is contained in $\cA_\theta$ and it uniquely extends to a continuous linear map $R:\cA_\theta'\rightarrow \cA_\theta$. We denote by $\Psi^{-\infty}(\T^n_\theta)$ the space of smoothing operators. 

\begin{proposition}[{\cite[Proposition~6.30]{HLP:Part1}}] \label{prop:PsiDOs.smoothing-operator-characterization}
A linear operator $R:\cA_\theta \rightarrow \cA_\theta$ is smoothing if and only if this  the \psido\ associated with some symbol in $\cS(\R^n; \cA_\theta)$.   
\end{proposition}

\begin{remark}
 As $\cS(\R^n;\cA_\theta)=\bigcap_{q\in \C}S^q(\R^n;\cA_\theta)$, we see that 
 $\Psi^{-\infty}(\T^n_\theta)= \bigcap_{q\in \C}\Psi^q(\T^n_\theta)$. 
\end{remark}

\subsection{Composition of \psidos}
Suppose we are given symbols $\rho_1(\xi)\in\stS^{m_1}(\Rn;\cA_\theta)$, $m_1\in\R$, and $\rho_2(\xi)\in\stS^{m_2}(\Rn;\cA_\theta)$, $m_2\in\R$. As $P_{\rho_1}$ and $P_{\rho_2}$ are linear operators on $\cA_\theta$, the composition $P_{\rho_1}P_{\rho_2}$ makes sense as such an operator.  
In addition, we define the map $\rho_1\sharp\rho_2:\Rn\rightarrow \cA_\theta$ by
\begin{equation} \label{eq:Composition.symbol-sharp}
\rho_1\sharp\rho_2(\xi) = (2\pi)^{-n}\iint e^{it\cdot\eta}\rho_1(\xi+\eta)\alpha_{-t}[\rho_2(\xi)]dt d\eta , \qquad \xi\in\Rn , 
\end{equation}
where the above integral is meant as an oscillating integral (see~\cite{Ba:CRAS88, HLP:Part2}). 

\begin{proposition}[see \cite{Ba:CRAS88, Co:CRAS80, HLP:Part2}] \label{prop:Composition.sharp-continuity-standard-symbol}
Let $\rho_1(\xi)\in \stS^{m_1}(\R^n; \cA_\theta)$ and  $\rho_2(\xi)\in \stS^{m_2}(\R^n; \cA_\theta)$, $m_1,m_2\in \R$. 
\begin{enumerate}
 \item $\rho_1\sharp\rho_2(\xi)\in \stS^{m_1+m_2}(\Rn;\cA_\theta)$, and we have $ \rho_1\sharp\rho_2(\xi) \sim \sum\frac{1}{\alpha !}\partial_\xi^\alpha\rho_1(\xi)\delta^\alpha\rho_2(\xi)$. 

 \item The operators $P_{\rho_1}P_{\rho_2}$ and $P_{\rho_1\sharp \rho_2}$ agree.           
\end{enumerate}
\end{proposition}

\begin{corollary}\label{cor:PsiDOs.composition-classical}
For $i=1,2$, let $P_i\in \Psi^{q_i}(\cA_\theta)$, $q_i\in \C$, have principal symbol $\rho_i(\xi)$. Then the composition  $P_1P_2$ is an operator in $\Psi^{q_1+q_2}(\cA_\theta)$ and has principal symbol $\rho_1(\xi)\rho_2(\xi)$. 
\end{corollary}

\subsection{Boundedness and infinitesimalness} 
A detailed account on spectral theoretic properties of \psidos\ is given in~\cite{HLP:Part2}. In this paper we will only need the following results. 

\begin{proposition}[see~\cite{HLP:Part2}] \label{prop:PsiDOs.boundedness} 
 Let $\rho(\xi)\in \stS^m(\R^n;\cA_\theta)$, $m\leq 0$. Then the operator $P_\rho$ uniquely extends to a bounded operator 
 $P_\rho:\cH_\theta \rightarrow \cH_\theta$. We obtain a compact operator when $m<0$. 
\end{proposition}

This result allows us to identify $\Psi^q(\T^n_\theta)$ with a subspace of $\cL(\cH_\theta)$ when $\Re q\leq 0$. 

\begin{proposition}[\cite{HLP:Part2}]\label{prop:continuity-weak-Schatten-class}
 Let $m<0$ and set $p=n|m|^{-1}$. 
\begin{enumerate}
 \item For every $\rho(\xi)\in \stS^{m}(\R^n;\cA_\theta)$ the operator $P_\rho$ is in $\scL^{p,\infty}$, i.e., this is an infinitesimal of order~$\leq \frac{1}{p}$. 
 
 \item This provides us with a continuous linear map $\stS^{m}(\R^n;\cA_\theta)\ni \rho(\xi)\rightarrow P_\rho\in \scL^{p,\infty}$.
 \end{enumerate} 
\end{proposition}

\begin{remark}
 As stated above Proposition~\ref{prop:continuity-weak-Schatten-class} is part of the contents of~\cite[Proposition~13.8]{HLP:Part2} for $p>1$ only. However, as pointed out in~\cite[Remark~13.12]{HLP:Part2}, the result holds \emph{verbatim} for any quasi-Banach ideal $\scI$ such that $\Delta^{\frac{m}2}\in \scI$, where $\Delta=\delta_1^2+\cdots +\delta_n^2$ is the flat Laplacian. As $\mu_k(\Delta^{\frac{m}2})=\op{O}(k^{\frac{m}{2}})=\op{O}(k^{-\frac1{p}})$, the result holds for $\scL^{p,\infty}$ with $p\leq 1$. 
\end{remark}

\begin{corollary}
 Every operator $P\in \Psi^{q}(\T^n_\theta)$ with $\Re q<-n$ is trace-class. 
\end{corollary}

\begin{corollary}
 Every operator $P\in \Psi^{-n}(\T^n_\theta)$ is in the weak trace-class $\scL^{1,\infty}$. 
\end{corollary}

\section{Noncommutative Residue}\label{sec:NCR}  
In this section, we survey the construction of the noncommutative residue trace for \psidos\ on NC tori. 

The noncommutative residue is defined on $\Psi^{\Z}(\T_\theta^n):=\bigcup_{q\in \Z}\Psi^q(\T_\theta^n)$.  Note that 
$\Psi^{\Z}(\T_\theta^n)$ is a sub-algebra of $\scL(\cA_\theta)$, and so by a trace we shall mean any linear functional that vanishes on commutators $[P,Q]$, $P,Q\in \Psi^\Z(\T^n_\theta)$. 

\begin{definition}
 The \emph{noncommutative residue} is the linear functional $\Res :\Psi^\Z(\T^n_\theta)\rightarrow \C$ given by
\begin{equation}
 \Res(P) = \int_{\bS^{n-1}} \tau\left[ \rho_{-n}(\xi)\right] d^{n-1}\xi, \qquad P\in \Psi^{\Z}(\cA_\theta),
 \label{eq:NCR.NCR}
\end{equation}
where $\rho_{-n}(\xi)$ is the symbol of degree~$-n$ of $P$. 
\end{definition}

\begin{remark}
 If $P\in \Psi^\Z(\T_\theta^n)$, then the homogeneous symbol $\rho_{-n}(\xi)$ is uniquely determined by $P$ 
 (\emph{cf}.~Remark~\ref{rem:PsiDOs.symbol-uniqueness}). We make the convention that $\rho_{-n}(\xi)=0$ when $P$ has order~$<-n$. 
\end{remark}

\begin{remark}
 It is immediate from the above definition that $\Res(P)=0$ when its symbol of degree~$-n$ is zero. In particular, the noncommutative residue vanishes on the following classes of operators:
\begin{enumerate}
 \item[(i)] \psidos\ of order~$<-n$, including smoothing operators. 
 
 \item[(ii)] Differential operators (since for such operators the symbols are polynomials without any homogeneous component of negative degree). 
\end{enumerate}
\end{remark}

\begin{remark}
In dimension $n=2$ (resp., $n=4$) we recover the noncommutative residue of~\cite{FW:JPDOA11} (resp., \cite{FK:JNCG15}). 
\end{remark}

\begin{remark}
 Under the equivalence between our classes of \psidos\ and the classes of toroidal \psidos\ (see~\cite{HLP:Part1}) the above noncommutative residue agrees with the noncommutative residue for toroidal \psidos\ introduced in~\cite{LNP:TAMS16}. 
\end{remark}

\begin{proposition}[\cite{FK:JNCG15, FW:JPDOA11, LNP:TAMS16}]\label{Prop2} Let $P\in \Psi^q(\cA_\theta)$ and $Q\in \Psi^{q'}(\cA_\theta)$ be such that $q+q'\in \Z$. Then 
\begin{equation*}
 \Res(PQ)= \Res(QP). 
\end{equation*}
 In particular, the noncommutative residue $\Res$ is a trace on the algebra $\Psi^{\Z}(\cA_\theta)$. 
\end{proposition}

By a well-known result of Wodzicki~\cite{Wo:HDR}, on connected closed manifolds of dimension $n \geq 2$ the noncommutative residue is the unique trace up to constant multiple. The following is the analogue of this result for NC tori (see also~\cite{FW:JPDOA11, FK:JNCG15, LNP:TAMS16}). 

\begin{proposition}[\cite{Po:TracesNCT}]
 Every trace on $\Psi^\Z(\T^n_\theta)$ is a constant multiple of the noncommutative residue. 
\end{proposition}
 
 We also observe that the formula~(\ref{eq:NCR.NCR}) can be rewritten in the form, 
\begin{equation}
 \Res(P) = \tau\big[ c_P\big], \qquad \text{where}\ c_P:=\int_{\bS^{n-1}} \rho_{-n}(\xi)d^{n-1}\xi. 
 \label{eq:NCR.cP}
\end{equation}
Note that if $P\in \Psi^m(\T^n_\theta)$, $m\in \Z$, has symbol $\rho(\xi) \sim \sum \rho_{m-j}(\xi)$, then, for every $a\in \cA_\theta$, the operator $aP$ has symbol $a\rho(\xi)\sim \sum a\rho_{m-j}(\xi)$. In particular, its symbol of degree~$-n$ is $a\rho_{-n}(\xi)$, and so we have $c_{aP}=\int_{\bS^{n-1}} a\rho_{-n}(\xi)d^{n-1}\xi= ac_{P}$. Thus, for all $a\in \cA_\theta$, we have 
\begin{equation}
 \Res(aP)= \tau\big[ c_{aP}\big]=  \tau\big[ ac_P\big].
 \label{eq:NCR.local}
\end{equation}
This shows that the noncommutative residue is a local functional.

\section{Derivatives and Sums of Commutators}\label{sec:deltaj-derivatives}  
In this section, we establish that any \psido\ agrees with a scalar-symbol \psido\ modulo the closure of the weak trace class commutator subspace. 
The approach is based on the following observation. 

\begin{proposition}[\cite{Po:TracesNCT}]\label{prop:derivatives.sums} 
 Let $\rho(\xi)\in \stS^{m}(\R^n;\cA_\theta)$, $m\in \R$. Then there are symbols $\rho_1(\xi), \ldots, \rho_n(\xi)$ in $\stS^m(\R^n;\cA_\theta)$ such that
\begin{equation}
\rho(\xi) = \tau\big[ \rho(\xi)\big] + \delta_1\rho_1(\xi) + \cdots + \delta_n\rho_n(\xi).
\label{eq:derivatives.sums-derivatives} 
\end{equation}
\end{proposition}

\begin{remark}
 For instance in~(\ref{eq:derivatives.sums-derivatives}) we may take $\rho_{j}(\xi)=\delta_j \Delta^{-1} \rho(\xi)$, where $\Delta^{-1}$ is the partial inverse of the (flat) Laplacian $\Delta := (\delta_1^2 +\cdots +\delta_n^2)$. Furthermore, if $\rho(\xi) \in S^m(\R^n;\cA_\theta)$, then the above symbols are in $S^m(\R^n;\cA_\theta)$ as well (see~\cite{Po:TracesNCT}). 
\end{remark}

It can be shown that  if $\rho(\xi)\in \stS^{m}(\R^n;\cA_\theta)$, then $P_{\delta_j\rho}=[\delta_j,P_\rho]$ (see~\cite{Po:TracesNCT}). When $m\leq -n$ this does not imply that the operators $P_{\delta_j \rho}$ lie in the commutator space $\Com(\cL^{1,\infty})$, because the derivatives $\delta_j$ are unbounded operators. Nevertheless, as the following result asserts, those operators are contained in the $\cL^{1,\infty}$-closure $\overline{\Com(\cL^{1,\infty})}$ of $\Com(\cL^{1,\infty})$.

\begin{proposition}\label{prop:derivatives.ComL1infty} 
 Let $\rho(\xi)\in \stS^{-n}(\R^n;\cA_\theta)$. 
\begin{enumerate}
 \item For $j=1,\ldots, n$, we have
\begin{equation*}
 P_{\delta_j\rho}  = \lim_{t \rightarrow 0} \frac{1}{t}\left( \alpha_{te_j}P_\rho  \alpha_{te_j}^{-1}-P_\rho \right) \quad \text{in $\cL^{1,\infty}$}.
\end{equation*}

\item The operators $P_{\delta_1 \rho}, \ldots, P_{\delta_n \rho}$ are contained in $\overline{\Com(\cL^{1,\infty})}$. 
\end{enumerate}
\end{proposition}
\begin{proof}
 Let $j\in \{1,\ldots, n\}$. Given any $\xi\in \R^n$, the map $\R\ni t \rightarrow \alpha_{te_j}[\rho(\xi)]\in \cA_\theta$ is smooth and  $\frac{d^\ell}{dt^\ell}\alpha_{te_j}[\rho(\xi)]= i^\ell \alpha_{t e_j}[\delta_j^\ell\rho(\xi)]$ for all $\ell\geq 1$. Thus, by  the Taylor formula for maps with values in locally convex spaces (see, e.g., \cite[Proposition~C.15]{HLP:Part1}), for all $t\in \R$, we have 
\begin{equation*}
 \alpha_{te_j}\big[\rho(\xi)\big] - \rho(\xi)-t\delta_j\rho(\xi)= -t^2\int_0^1(1-s)  \alpha_{se_j}\left[\delta_j^2\rho(\xi)\right]ds. 
\end{equation*}
For $t\neq 0$ and $\xi\in \R^n$ set $\rho_t(\xi)=t^{-1}(\alpha_{te_j}[\rho(\xi)] - \rho(\xi))$. We have
\begin{align*}
 \big\| \rho_t(\xi)-\delta_j\rho(\xi)\big\| & \leq t \int_0^1(1-s)  \left\|\alpha_{se_j}\left[\delta_j^2\rho(\xi)\right]\right\|ds
 \\
 & \leq t \int_0^1(1-s)  \left\|\delta_j^2\rho(\xi)\right\|ds\\ 
& \leq \frac{1}{2} t  \left\|\delta_j^2\rho(\xi)\right\|. 
\end{align*}
As $\rho(\xi) \in  \stS^{-n}(\R^n; \cA_\theta)$ there is $C>0$ such that $\left\|\delta_j^2\rho(\xi)\right\|\leq C(1+|\xi|)^m$ for all $\xi\in \R^n$. Thus, there is $C>0$ such that, for all $\xi\in \R^n$ and $t\in \R\setminus 0 $, we have 
\begin{equation*}
  \big\| \rho_t(\xi)-\delta_j\rho(\xi)\big\| \leq Ct(1+|\xi|)^{-n}. 
\end{equation*}

Likewise, given any multi-orders $\alpha$ and $\beta$, as $\delta^\alpha \partial_\xi^\beta \rho(\xi)\in \stS^{m-|\beta|}(\R^n; \cA_\theta)$, there is $C_{\alpha\beta}>0$ such that, for for all $\xi\in \R^n$ and $t\in \R\setminus 0 $, we have 
\begin{align*}
 \big\| \delta^\alpha \partial_\xi^\beta\left[\rho_t(\xi)-\delta_j\rho(\xi)\right]\big\| & =
  \bigg\| \frac1{t}\left(\alpha_{te_j}\big[\delta^\alpha \partial_\xi^\beta\rho(\xi)\big] - \delta^\alpha \partial_\xi^\beta\rho(\xi)\right)-\delta_j\delta^\alpha \partial_\xi^\beta\rho(\xi)\bigg\| \\
  & \leq Ct(1+|\xi|)^{-n-|\beta|}. 
\end{align*}
This implies that, for all multi-orders $\alpha$ and $\beta$, we have 
\begin{equation*}
\lim_{t\rightarrow 0} \sup_{\xi\in \R^n} (1+|\xi|)^{n+|\beta|} \big\| \delta^\alpha \partial_\xi^\beta\left[\rho_t(\xi)-\delta_j\rho(\xi)\right]\big\| =0. 
\end{equation*}
This shows that $\rho_t(\xi) \rightarrow \delta_j\rho(\xi)$ in $\stS^{-n}(\R^n;\cA_\theta)$ as $t\rightarrow 0$.  Combining this with 
Proposition~\ref{prop:continuity-weak-Schatten-class} shows that, in $\cL^{1,\infty}$, we have
\begin{equation*}
P_{\delta_j\rho} = \lim_{t \rightarrow 0} P_{\rho_t} = \lim_{t \rightarrow 0} \frac{1}{t}\left(P_{\alpha_{tej}(\rho)}-P_\rho \right).
\label{eq:derivatives.limit-conjugation-alpha}   
\end{equation*}

We claim that, for all $t\in \R$, we have
\begin{equation}
 P_{\alpha_s(\rho)} =\alpha_{te_j} P_{\rho} \alpha_{te_j}^{-1}.
 \label{eq:derivatives.conjugation-alpha} 
\end{equation}
To see this set $s=te_j$, and let $u= \sum_{k\in \Z^n} u_k U^k$ be in $\cA_\theta$. As $\alpha_{s}$ is a continuous automorphism of $\cA_\theta$ we have 
\begin{equation}
 (\alpha_s P_\rho \alpha_s^{-1})u= \sum_{k\in \Z^n} u_k (\alpha_s P_\rho \alpha_s^{-1})U^k.
 \label{eq:derivative.alphas-conjugation} 
\end{equation}
As $\alpha_s^{-1}(U^k)= \alpha_{-s}(U^k)= e^{-is\cdot k} U^k$ we get
\begin{equation*}
 P_\rho \left[\alpha_s^{-1}\big(U^k\big) \right]= e^{-is\cdot k} P_{\rho}(U^k)= e^{-is\cdot k} P_{\rho}(U^k) = e^{-is\cdot k} \rho(k)U^k= \rho(k)\alpha_s^{-1}\big(U^k\big). 
\end{equation*}
Thus,
\begin{equation*}
 \big( \alpha_s P_\rho \alpha_{s}^{-1}\big)\big(U^k\big)= \alpha_s\left[ \rho(k)\alpha_s^{-1}\big(U^k\big)\right] =  \alpha_s\left[ \rho(k)\right]U^k. 
\end{equation*}
Combining this with~(\ref{eq:derivative.alphas-conjugation}) gives
\begin{equation*}
 \big(\alpha_s P_\rho \alpha_s^{-1}\big)u= \sum_{k\in \Z^n} u_k \alpha_s\left[ \rho(k)\right]U^k=P_{\alpha_s(\rho)}u. 
\end{equation*}
This proves~(\ref{eq:derivatives.conjugation-alpha}). 

Combining~(\ref{eq:derivatives.limit-conjugation-alpha}) and~(\ref{eq:derivatives.conjugation-alpha}) proves the first part of the proposition. The 2nd part follows from the fact that 
$\alpha_{te_j}P_\rho  \alpha_{te_j}^{-1}-P_\rho= [\alpha_{te_j},P_\rho  \alpha_{te_j}^{-1}]\in \op{Com}(\cL^{1,\infty})$. The proof is complete.  
\end{proof}

\section{WeylÕs law for Scalar-Valued Symbols} \label{sec:Weyl} 
In this section, we establish a WeylÕs law for positive \psidos\ with scalar symbols. This will allows us to get a preliminary version of the trace theorem for \psidos\ with 
scalar-symbols. 

Given any homogeneous symbol $\rho(\xi)\in S_m(\R^n;\cA_\theta)$, $m\in \R$, we define the linear operator $P_\rho: \cA_\theta \rightarrow \cA_\theta$ by
\begin{equation*}
 P_\rho = \sum_{k\in \Z^n\setminus 0} u_k \rho(k) U^k, \qquad u=\sum u_kU^k\in \cA_\theta. 
\end{equation*}
In fact, $P_\rho=P_{\tilde{\rho}}$, with $\tilde{\rho}(\xi)=(1-\chi(\xi))\rho(\xi)\in S^m(\R^n;\cA_\theta)$, where $\chi(\xi)$ is any function in $C^\infty_c(\R^n)$ such that $\chi(\xi)=1$ near $\xi=0$ and $\chi(\xi)\geq 1$. In particular, we see that $P_\rho$ is an operator in $\Psi^{m}(\cA_\theta)$ and has $\rho(\xi)$ as principal symbol. 

Suppose now that $\rho(\xi)$ has scalar values, i.e., $\rho(\xi)\in S_m(\R^n)$. In this case, the operator $P_\rho$ is even \emph{diagonal} with respect to the orthonormal basis $(U^k)_{k\in \Z^n}$. Namely, 
\begin{equation*}
 P_\rho(1)=0, \qquad P_\rho(U^k)=\rho(k)U^k, \quad k\in \Z^n\setminus 0. 
\end{equation*}
Therefore, the operator is (formally) normal and its (non-zero) eigenvalues are given by the values of the symbol $\rho(\xi)$ on $\Z^n\setminus 0$. 

Assume further that $m>0$ and $\rho(\xi)>0$. This implies that $P_\rho$ is elliptic and $\rho(k)>0$ for all $k\in \Z^n\setminus 0$.  It is convenient to extend $\rho(\xi)$ by continuity to $\xi=0$ by setting $\rho(0)=0$. Then, the operator $P_\rho$ is a selfadjoint Fredholm operator with spectrum $\rho(\Z^n)$. Note also that the nullspace of $P_\rho$ is just $\C \cdot 1$. Therefore, we can list the eigenvalues of $P_\rho$ as a non-decreasing sequence, 
\begin{equation*}
 0=\lambda_0(P_\rho)<\lambda_1(P_\rho)\leq \lambda_2(P_\rho)\leq \cdots,
\end{equation*}
where each eigenvalue is repeated according to multiplicity. We also define the counting function, 
\begin{equation*}
 N(P_\rho;\lambda):= \#\left\{\ell\in \N_0;\ \lambda_\ell(P_\rho)\leq \lambda\right\}, \qquad \lambda\geq 0. 
\end{equation*}
Note that we have
\begin{equation}
  N(P_\rho;\lambda):= \#\left\{k\in \Z^n\setminus 0;\ \rho(k)\leq \lambda\right\}.
  \label{eq:Weyl.counting-symbol} 
\end{equation}

\begin{proposition}[WeylÕs law] \label{prop:Weyl1} 
Let $\rho(\xi)\in S_1(\R^n)$, $\rho(\xi)>0$.
\begin{enumerate}
 \item As $\lambda \rightarrow \infty$, we have
\begin{equation*}
 N(P_\rho;\lambda)=c(\rho)\lambda^n\left(1+\op{O}(\lambda^{-1})\right), \qquad c(\rho):= \frac{1}{n}\int_{\bS^{n-1}}\rho(\xi)^{-\frac1{n}}d^{n-1}\xi>0. 
\end{equation*}

\item As $\ell \rightarrow 0$, we have 
\begin{equation}
 \lambda_\ell(P_\rho)= \left(c(\rho)^{-1}\ell\right)^{\frac{1}{n}} + \op{O}(1). 
 \label{eq:Weyl.eigenvalues}
\end{equation}
\end{enumerate}
\end{proposition}
\begin{proof}
 The proof follows along similar lines as that of the classical proof of the Weyl law for the Laplacian on an ordinary torus. For $k\in \Z^n$ set $I(k)=k+[0,1]^n$. Note that 
\begin{equation*}
 |\xi-k|\leq \sqrt{n}\qquad \text{for all $\xi\in I(k)$}.
\end{equation*}
For $\lambda \geq 0$, we set 
\begin{equation*}
 A(\lambda)= \left\{\xi\in \R^n;\ \rho(\xi)<1\right\}, \qquad B(\lambda)=\bigcup_{k\in A(\lambda)\cap \Z^n} I(k). 
\end{equation*}
 We observe that by~(\ref{eq:Weyl.counting-symbol}) we have
\begin{equation}
  N(P_\rho;\lambda)=\#\left(A(\lambda)\cap \Z^n\right) = \sum_{k\in A(\lambda)\cap \Z^n} I(k) = \left| B(\lambda)\right|. 
  \label{eq:Weyl.counting-volume}
\end{equation}
 
 As $\rho(\xi)\in S_{1}(\R^n)$, the partial derivatives $\partial_{\xi_1}\rho(\xi), \ldots, \partial_{\xi_n}\rho(\xi)$ are symbols in $S_0(\R^n)$, and so they are bounded on $\R^n\setminus 0$. Therefore, there is $C>0$ such that 
\begin{equation}
 |\rho(\xi)-\rho(\eta)| \leq C|\xi-\eta| \qquad \forall \xi,\eta \in \R^n\setminus 0.
 \label{eq:Weyl.Taylor-rho} 
\end{equation}

Let $\lambda\geq 0$ and $k\in A(\lambda)\cap \Z^n$. The inequality~(\ref{eq:Weyl.Taylor-rho}) for $\xi\in I(k)$ and $\eta =k$ gives 
\begin{equation*}
 |\rho(\xi)|\leq |\rho(k)|+C|\xi-k|\leq \lambda +C\sqrt{n}. 
\end{equation*}
Therefore, we see that $B(\lambda)\subset A(\lambda +C\sqrt{n})$. 

Conversely,  suppose that $\lambda \geq C\sqrt{n}$ and let $\xi \in A(\lambda-C\sqrt{n})$ and $k\in \Z^n$ be such that $\xi\in I(k)$. Then by using~(\ref{eq:Weyl.Taylor-rho}) we get
\begin{equation*}
 |\rho(k)|\leq |\rho(\xi)| + C|\xi-k|\leq (\lambda -C\sqrt{n})+C\sqrt{n}=\lambda.
\end{equation*}
Thus, $k\in A(\lambda)\cap \Z^n$, and so $\xi \in B(\lambda)$. It then follows that $A(\lambda-C\sqrt{n})\subset B(\lambda)$. 

Combining the inclusions $A(\lambda-C\sqrt{n})\subset B(\lambda) \subset A(\lambda +C\sqrt{n})$ with~(\ref{eq:Weyl.counting-volume}) gives
\begin{equation*}
 \left| A(\lambda-C\sqrt{n})\right| \leq N(P_\rho;\lambda) \leq  \left| A(\lambda+C\sqrt{n})\right| \qquad \forall \lambda \geq C\sqrt{n}. 
\end{equation*}
Moreover, the homogeneity of $\rho(\xi)$ implies that, for $\lambda>0$, we have
\begin{equation*}
 \left|A(\lambda)\right|=\left|\left\{\xi\in \R^n;\ \rho(\lambda^{-1}\xi)<1\right\}\right|=\left|\lambda A(1)\right|=\lambda^n\left| A(1)\right|. 
\end{equation*}
Therefore, for $\lambda \geq C\sqrt{n}$ we have 
\begin{equation*}
 (\lambda-C\sqrt{n})^n\left|A(1)\right| \leq  N(P_\rho;\lambda) \leq  (\lambda-C\sqrt{n})^n\left|A(1)\right|.   
\end{equation*}
As $(\lambda \pm C\sqrt{n})^n=\lambda^n+ \op{O}(\lambda^{n-1})$ as $\lambda \rightarrow \infty$, we then deduce that
\begin{equation}
  N(P_\rho;\lambda) = \left|A(1)\right| \lambda^n\left(1+  \op{O}(\lambda^{-1})\right).
  \label{eq:Weyl.couting-Weyl-Alambda} 
\end{equation}

It remains to compute  $|A(1)|$. Thanks to the homogeneity of $\rho(\xi)$,  in polar coordinates $\xi=r\eta$ the inequality $\rho(\xi)\leq 1$ gives $1\geq \rho(r\eta) =r\rho(\eta)$, i.e., $r<\rho(\eta)^{-1}$. Therefore, by integrating in polar coordinates we get
\begin{equation*}
  \left|A(1)\right|=\int_{\rho(\xi)\leq 1} d\xi = 
   \int_{\bS^{n-1}} \left(\int_0^{\rho(\eta)^{-1}} r^{n-1}dr\right) d^{n-1}\eta= c(\rho), 
\end{equation*}
where we have set $c(\rho)= \frac{1}{n} \int_{\bS^{n-1}}\rho(\eta)^{-\frac1{n}}d^{n-1}\eta$. Note that $c(\rho)>0$ since $\rho(\xi)>0$ on $\bS^{n-1}$. Combining this with~(\ref{eq:Weyl.couting-Weyl-Alambda}) gives the WeylÕs law, 
\begin{equation}
  N(P_\rho;\lambda) =c(\rho)\lambda^n\left(1+ \op{O}(\lambda^{-1})\right) \qquad \text{as $\lambda\rightarrow \infty$}.
  \label{eq:Weyl.counting} 
\end{equation}

It is routine to deduce from~(\ref{eq:Weyl.counting}) the WeylÕs law~(\ref{eq:Weyl.eigenvalues}) for the eigenvalues $\lambda_\ell(P_\rho)$. Namely, as~(\ref{eq:Weyl.counting}) implies that 
$N(P_\rho;\lambda) =c(\rho)\lambda^n$, we see that $\lambda\sim (c(\rho)^{-1}N(P_\rho;\lambda))^{\frac{1}{n}}$. This allows us to rewrite~(\ref{eq:Weyl.counting}) in the form, 
\begin{equation*}
 \lambda^n=c(\rho)^{-1}N(P_\rho;\lambda)\left[1 +\op{O}\left(N(P_\rho;\lambda)^{-\frac1{n}}\right) \right] \qquad \text{as $\lambda\rightarrow \infty$}. 
\end{equation*}
Equivalently, as $\lambda\rightarrow \infty$, we have
\begin{equation}
 \lambda = \left(c(\rho)^{-1}N(P_\rho;\lambda)\right)^{\frac{1}{n}}\left[1 +\op{O}\left(N(P_\rho;\lambda)^{-\frac1{n}}\right) \right] = \left(c(\rho)^{-1}N(P_\rho;\lambda)\right)^{\frac{1}{n}} +\op{O}(1).
 \label{eq:Weyl.counting-lambda} 
\end{equation}

Note that, for all $\ell\geq 0$ and $\epsilon>0$, we have $N(P_\rho;\lambda_\ell(P_\rho)-\epsilon)\leq \ell \leq N(P_\rho;\lambda_\ell(P))$. Combining this with~(\ref{eq:Weyl.counting-lambda} ) we see that, as $\ell \rightarrow \infty$, we have
\begin{gather*}
 \lambda_\ell(P_\rho) = \left[c(\rho)^{-1}N\left(P_\rho;\lambda_\ell(P_\rho)\right)\right]^{\frac{1}{n}} +\op{O}(1)\leq \left(c(\rho)^{-1}\ell \right)^{\frac{1}{n}} +\op{O}(1),\\
  \lambda_\ell(P_\rho)=  \lambda_\ell(P_\rho)-\epsilon +\op{O}(1) = \left[c(\rho)^{-1}N\left(P_\rho;\lambda_\ell(P_\rho)-\epsilon\right)\right]^{\frac{1}{n}} +\op{O}(1)
  \geq \left(c(\rho)^{-1}\ell \right)^{\frac{1}{n}} +\op{O}(1). 
\end{gather*}
This gives the WeylÕs law~(\ref{eq:Weyl.eigenvalues}). The proof is complete.  
\end{proof}

In what follows we let $T_0$ be a selfadjoint operator in $\cL^{1,\infty}$ whose eigenvalue sequence is the harmonic sequence, i.e., there is an orthonormal basis $(e_\ell)_{\ell \leq 0}$ of $\cH_\theta$ such that $T_0 e_\ell =(\ell+1)^{-1}e_\ell$ for all $\ell\geq 0$. 

\begin{proposition}\label{prop:Weyln} 
 For every $\rho(\xi)\in S_{-n}(\R^n)$, we have
\begin{equation}
 P_\rho= \frac{1}{n} \Res (P_\rho) T_0 \qquad \bmod\  \op{Com}(\cL^{1,\infty}). 
\label{eq:Scalar-commutator} 
\end{equation}
\end{proposition}
\begin{proof}
Let $\rho(\xi)\in S_{-n}(\R^n)$. As mentioned above, $P_\rho$ is an operator in $\Psi^{-n}(\cA_\theta)$ with principal symbol $\rho(\xi)$, and so $\Res(P_\rho)=\int_{\bS^{n-1}}\rho(\xi) d^{n-1}\xi$. 

We observe that both sides of~(\ref{eq:Scalar-commutator}) depends linearly on the symbol $\rho(\xi)$. Upon writing $\rho(\xi)=\frac12(\rho(\xi)+\overline{\rho(\xi)})+\frac1{2i}(\rho(\xi)-\overline{\rho(\xi)})$ we see it is enough to establish~(\ref{eq:Scalar-commutator}) when $\rho(\xi)$ is real-valued. 

Assume that $\rho(\xi)$ is real-valued and is not identically zero. Set $c=\sup\{|\rho(\xi)|;\ |\xi|=1\}$. By homogeneity $|\rho(\xi)|=|\xi|^{-n} \rho(|\xi|^{-1}\xi)\leq c|\xi|^{-n}$ for all $\xi \in \R^n\setminus 0$. In particular, we see that $c>0$, and so $\rho(\xi)+2c|\xi|^{-n}\geq c|\xi|^{-n}>0$ for $\xi\neq 0$. As $\rho(\xi)=(\rho(\xi)+2c|\xi|^{-n})-2c|\xi|^{-n}$, we see that $\rho(\xi)$ is a linear combination of positive symbols in $S_{-n}(\R^n)$. Together with the linearity of~(\ref{eq:Scalar-commutator}) this further reduces the verification of~(\ref{eq:Scalar-commutator}) to the case of positive symbols. 

Suppose now that $\rho(\xi)>0$. Set $\sigma(\xi)=\rho(\xi)^{-\frac1n}$. Then $\sigma(\xi)\in S_1(\R^n)$ and $\rho(\xi)=\sigma(\xi)^{-n}$. This implies that $P_\rho (U^k)=\sigma (k)^{-n}U^k$ for all $k\in \Z^n\setminus 0$, and so $P_\rho =P_\sigma^{-n}$. It then follows that $\mu_\ell(P_\rho)=\lambda_{\ell +1}(P_\rho)^{-n}$ for all $\ell\geq 0$. Combining this with the WeylÕs law~(\ref{eq:Weyl.eigenvalues}) then shows that, as $\ell \rightarrow \infty$, we have
\begin{equation*}
\mu_\ell(P_\rho) = \lambda_{\ell +1}(P_\rho)^{-n}=\left[ c(\sigma)^{-1}\ell^{\frac{1}{n}}\left(1+\op{O}(\ell^{-\frac1n})\right)\right]^{-n} = c(\sigma)\frac{1}{\ell}\left(1+\op{O}(\ell^{-\frac1n})\right). 
\end{equation*}
Here $c(\sigma)=\frac{1}{n} \int_{\bS^{n-1}}\sigma(\xi)^{-n}d^{n-1}\xi= \frac{1}{n} \int_{\bS^{n-1}}\rho(\xi)d^{n-1}\xi= \Res(P_\rho)$.  Thus,
\begin{equation}
 \mu_\ell(P_\rho)= \frac{1}{n}\Res (P_\rho) \frac{1}{\ell} + \op{O}\left(\ell^{-(1+\frac1n)}\right).
 \label{eq:Weyl.Weyl-law-mul}
\end{equation}

Recall there is an orthonormal basis $(e_\ell)_{\ell \leq 0}$ of $\cH_\theta$ such that $T_0 e_\ell =(\ell+1)^{-1}e_\ell$ for all $\ell\geq 0$. Let $(v_\ell)_{\ell\geq 0}$ be a re-arrangement of the sequence $(U^k)_{k\in \Z}$ such that $P_\rho v_\ell= \mu_\ell(P_\rho)v_\ell$ for all $\ell\geq 0$. In addition, let $V\in \cL(\sH)$ be the unitary operator such that $Vv_\ell=e_\ell$ for all $\ell \geq 0$, and set $T_\rho=\frac{1}n\Res (P_\rho)V^*T_0V$. We have
\begin{equation}
 T_\rho-\frac{1}n\Res (P_\rho)T_0= \frac{1}n\Res (P_\rho) (V^*T_0V-V)=  \frac{1}n\Res (P_\rho) [V^*,T_0V]\in  \Com(\cL^{1,\infty}). 
 \label{eq:Weyl.T0-Trho}
\end{equation}
Moreover, $T_\rho v_\ell= \frac{1}n\Res (P_\rho)V^*T_0e_\ell= \frac{1}n\Res (P_\rho)(\ell+1)^{-1} v_\ell$. 
It then follows that the operator $P_\rho-T_\rho$ is diagonal with respect to the orthonormal basis $(v_\ell)_{\ell\geq 0}$ with eigenvalues $\mu_\ell(P_\rho)-\frac{1}n\Res (P_\rho)(\ell+1)^{-1}$, $\ell \geq 0$. Combining this with~(\ref{eq:Weyl.Weyl-law-mul}) gives
\begin{equation*}
 \Tr\left( |P_\rho-T_\rho|\right) = \sum_{\ell \geq 0} \left| \mu_\ell(P_\rho)-\frac{1}{n(\ell+1)}\Res (P_\rho)\right|<\infty. 
\end{equation*}
This shows that $P_\rho-T_\rho$ is a trace-class operator, and hence is contained in $\Com(\cL^{1,\infty})$ by Proposition~\ref{prop:Quantized.trace-class-commutators}. Combining this with~(\ref{eq:Weyl.T0-Trho}) we then deduce that $P_\rho-T_\rho\in \Com(\cL^{1,\infty})$. This proves~(\ref{eq:Scalar-commutator}) when $\rho(\xi)$ is positive. The proof is complete. 
\end{proof}

\section{Connes' trace theorem for Noncommutative Tori}\label{sec:Connes-Trace}  
In this section, we prove the version for noncommutative tori of Connes' trace theorem and give a couple of applications. 

\begin{theorem}\label{thm:Trace-Thm} 
Every operator $P\in \Psi^{-n}(\cA_\theta)$ is strongly measurable, and we have
\begin{equation}
 \bint P= \frac{1}{n} \Res (P).
 \label{eq:Trace-Thm.trace-formula} 
\end{equation}
\end{theorem}
\begin{proof}
 Let $P\in \Psi^{-n}(\cA_\theta)$ have symbol $\rho(\xi)\in S^{-n}(\R^n;\cA_\theta)$ and principal symbol $\rho_{-n}(\xi)\in S_{-n}(\R^n;\cA_\theta)$. Set $\tilde{\rho}_{-n}(\xi) =(1-\chi)\in C^\infty_c(\R^n)$ is equal to $1$ near $\xi=0$ and vanishes for $|\xi|\geq 1$. Then $\tilde{\rho}_{-n}(\xi)\in S^{-n}(\R^n)$ and $\rho(\xi) -\tilde{\rho}_{-n}(\xi)\in S^{-n-1}(\R^n)$. As operators in $\Psi^{-n-1}(\cA_\theta)$ are trace-class, and hence are contained in $\Com(\cL^{1,\infty})$ by Proposition~\ref{prop:Quantized.trace-class-commutators}, we deduce that
\begin{equation}
 P=P_\rho=P_{\tilde{\rho}_{-n}} \qquad \bmod \Com(\cL^{1,\infty}).
 \label{eq:Trace-Thm.P-Prhon}  
\end{equation}

By applying Proposition~\ref{prop:derivatives.sums} and Proposition~\ref{prop:derivatives.ComL1infty} to $\tilde{\rho}_{-n}(\xi)$ we see that there are symbols $\rho_1(\xi), \ldots,\rho_n(\xi)$ in $\stS^{-n}(\R^n;\cA_\theta)$ such that 
\begin{equation}
 P_{\tilde{\rho}_{-n}} = P_{\tau[\tilde{\rho}_{-n}]} + P_{\delta_1\rho_1}+ \cdots + P_{\delta_n\rho_n}= P_{\tau[\tilde{\rho}_{-n}]} \quad \bmod \overline{\Com(\cL^{1,\infty})},
 \label{eq:Trace-Thm.Prhon-Com}
\end{equation}
 where $\overline{\Com(\cL^{1,\infty})}$ is the closure in $\cL^{1,\infty}$ of $\Com(\cL^{1,\infty})$. 
 Note that $\tau[\tilde{\rho}_{-n}(\xi)]=(1-\chi(\xi))\tau[\rho_{-n}(\xi)]$, and so $P_{\tau[\tilde{\rho}_{-n}]}= P_\tau[\rho_{-n}]$. 
 Combining this with~(\ref{eq:Trace-Thm.P-Prhon}) and~(\ref{eq:Trace-Thm.Prhon-Com}) we get 
\begin{equation}
 P= P_{\tau[\rho_{-n}] } \qquad \bmod \overline{\Com(\cL^{1,\infty})}.
 \label{eq:Trace-Thm.Ptaurhon-Com} 
\end{equation}
 
Note that $\tau[\rho_{-n}(\xi)]\in S^{-n}(\R^n)$ and $\Res(P_{\tau[\rho{-n}]})= \int_{\bS^{n-1}}\tau[\rho_{-n}(\xi)]d^{n-1}\xi=\Res(P)$. Therefore, by using Proposition~\ref{prop:Weyln} we see that $P_{\tau[\rho_{-n}]}-n^{-1}\Res(P)T_0\in \Com(\cL^{1,\infty})$. Combining this with~(\ref{eq:Trace-Thm.Ptaurhon-Com}) then gives 
\begin{equation}
 P= \frac{1}{n} \Res (P) T_0  \qquad \bmod \overline{\Com(\cL^{1,\infty})}.
 \label{eq:Trace-Thm.P-Com} 
\end{equation}
 
 Now, let $\varphi$ be positive normalized trace on $\cL^{1,\infty}$. By Proposition~\ref{prop:Quantized.continuity-positive-trace} this is a continuous trace, and so it is annihilated by $\overline{\Com(\cL^{1,\infty})}$. Combining this with~(\ref{eq:Trace-Thm.P-Com}) and the normalization $\varphi(T_0)=1$ we then obtain 
\begin{equation*}
\varphi(P)=   \frac{1}{n} \Res (P) \varphi(T_0)= \frac{1}{n} \Res (P). 
\end{equation*}
This proves that $P$ is strongly measurable and $\bint P=\frac1{n} \Res(P)$. The proof is complete. 
 \end{proof}

\begin{remark}
Fathizadeh-Khalkhali~\cite{FK:LMP13, FK:JNCG15} obtained a version of Theorem~\ref{thm:Trace-Thm} in dimension~$n=2$ and $n=4$ by a different approach. They established measurability and derived the trace formula~(\ref{eq:Trace-Thm.trace-formula}) in those dimensions, but they did not establish \emph{strong} measurability. That is, they established the trace formula~(\ref{eq:Trace-Thm.trace-formula}) for Dixmier traces only. 
\end{remark}

We shall now explain how Theorem~\ref{thm:Trace-Thm} allows us to recover and extend Connes' integration formula for (flat) NC tori of McDonald-Sukochev-Zanin~\cite[Theorem~6.15]{MSZ:MA19}. First, we have the following consequence of Theorem~\ref{thm:Trace-Thm}. 

\begin{corollary}\label{cor:Trace-thm.super-integration-formula} 
 Let $P\in \Psi^{-n}(\T^n_\theta)$. For every $a\in A_\theta$, the operator $aP$ is strongly measurable, and we have
\begin{equation*}
 \bint aP = \frac{1}{n}\tau(a c_P), 
 \end{equation*}
 where $c_P\in \cA_\theta$ is defined as in~(\ref{eq:NCR.cP}). 
\end{corollary}
\begin{proof}
We need to show that, for every normalized positive trace $\varphi$ on $\scL^{1,\infty}$, we have
 \begin{equation}
 \varphi\big( aP\big) = \frac{1}{n} \tau(a c_P) \qquad \forall a\in A_\theta,
 \label{eq:Trace-Thm.integration-phi}  
\end{equation}
When $a\in \cA_\theta$, then by combining the trace formula~(\ref{eq:Trace-Thm.trace-formula}) with~(\ref{eq:NCR.local}) we get
\begin{equation*}
 \varphi(aP)=\bint a P=\frac{1}{n} \Res (P)=  \frac{1}{n} \tau(a c_P). 
\end{equation*}
This shows that~(\ref{eq:Trace-Thm.integration-phi}) holds for all $a\in \cA_\theta$. 

As $\varphi$ is a positive trace,  Proposition~\ref{prop:Quantized.continuity-positive-trace} ensures that $\varphi$ is a continuous linear form on $\scL^{1,\infty}$. Note also that $A_\theta$ acts continuously on $\scL^{1,\infty}$ by left composition, since this action is the composition of the left regular representation of $A_\theta$ in $\cL(\cH_\theta)$ with the natural left-action of $\cL(\cH_\theta)$ on $\scL^{1,\infty}$. It then follows that both sides of~(\ref{eq:Trace-Thm.integration-phi}) are given by continuous linear forms on $A_\theta$. As~(\ref{eq:Trace-Thm.integration-phi}) holds on the dense subspace $\cA_\theta$ it holds on all $A_\theta$. The proof is complete. 
\end{proof}

Let $\Delta=\delta_1^2+\cdots+\delta_n^2$ be flat Laplacian on $\cA_\theta$. In the notation of Section~\ref{sec:Weyl} we have $\Delta=P_{\sigma}$ with $\sigma(\xi)=|\xi|^2$ and $\Delta^{-\frac{n}2}=P_\rho$ with $\rho(\xi)=|\xi|^{-n}$. In particular, $\Delta^{-\frac{n}2}$ is contained in $\Psi^{-n}(\cA_\theta)$ and has principal symbol $|\xi|^{-n}$, and so we have 
\begin{equation*}
 c_{\Delta^{-\frac{n}{2}}}=\int_{\bS^{n-1}} |\xi|^{-n}=|\bS^{n-1}|
\end{equation*}
Therefore, by specializing Corollary~\ref{cor:Trace-thm.super-integration-formula} to $P=\Delta^{-\frac{n}{2}}$ we recover the Connes integration formula for flat NC tori of~\cite{MSZ:MA19} in the following form. 

\begin{corollary}[{\cite[Theorem~6.15]{MSZ:MA19}}] \label{cor:Trace-thm.flat-integration-formula}
 For every $a\in A_\theta$ the operator $a\Delta^{-\frac{n}{2}}$ is strongly measurable, and we have
\begin{equation}
 \bint a\Delta^{-\frac{n}{2}} = \bcn  \tau(a), \qquad \bcn:= \frac{1}{n} |\bS^{n-1}|. 
 \label{eq:Trace-Thm.Integration-Formula}
\end{equation}
\end{corollary}

\begin{remark}
 The apparent discrepancy between the constant $c_n$ in~(\ref{eq:Quantized.integration-formula} ) and the constant $\bcn$ above is due to the fact when $\theta=0$ the trace $\tau$ is actually the Radon measure $a\rightarrow (2\pi)^{-n} \int_{\T^n}a(x)dx$. That is, the overall constant $(2\pi)^{-n}$ is absorbed into the definition of $\tau$. 
\end{remark}

\begin{remark}
 When $\theta=0$ we recover the integration formula~(\ref{eq:Quantized.integration-formula}) for the ordinary torus $\T^n$ equipped with its standard flat metric.
\end{remark}

\section{Riemannian Metrics and Laplace-Beltrami Operator} \label{sec:Riemannian}
In this section, we review the main definitions and properties regarding Riemannian metrics and Laplace-Beltrami operators. 

\subsection{Riemannian metrics} In what follows we denote by $\GL_n(\cA_\theta)$ the group of invertible matrices in $M_n(\cA_\theta)$. We also denote by 
$\GL_n^+(\cA_\theta)$ its subset of positive matrices. Here by a positive element of $M_n(\cA_\theta)$ we mean a selfadjoint matrix with non-negative spectrum. As $M_n(\cA_\theta)$ is closed under holomorphic functional calculus we have 
\begin{equation*}
 \GL_n^+(\cA_\theta)=\big\{h^*h; \ h \in h\in \GL_n(\cA_\theta)\big\}=\big\{h^2; \ h \in \GL_n(\cA_\theta), \ h^*=h\big\}.
\end{equation*}
In particular, if $h \in  \GL_n^+(\cA_\theta)$, then $h^{-1}$ and $\sqrt{h}$ are in $\GL_n^+(\cA_\theta)$. When $n=1$ we denote by $\cA_\theta^{++}$ the cone of positive invertible elements of $\cA_\theta$. 

Let $\sX_\theta$ be the free left-module over $\cA_\theta$ generated by the canonical derivations $\delta_1,\ldots, \delta_n$. This plays the role of the module of (complex) vector fields on the NC torus $\cA_\theta$ (\emph{cf}.~\cite{Ro:SIGMA13}). As  $\sX_\theta$ is a free module, the Hermitian metrics on $\sX_\theta$ are in one-to-one correspondence with matrices $h\in \GL^+_n(\cA_\theta)$ (see~\cite{HP:Laplacian, Ro:SIGMA13}). Namely, to any matrix $h=(h_{ij})$ in $ \GL^+_n(\cA_\theta)$ corresponds the Hermitian metric, 
\begin{equation*}
 \acoupt{X}{Y}_h := \sum_{1\leq i,j\leq n}  X_ih_{ij} Y_j^*, \qquad X= \sum_i X_i \delta_i, \quad Y= \sum_j Y_j \delta_j. 
\end{equation*}

On a $C^\infty$-manifold a Riemannian metric is a Hermitian metric on complex vector fields which takes real-values on real vector fields and real differential forms. Equivalently, its matrix and its inverse have real entries. On the NC torus $\T_\theta^n$ the real-valued smooth functions corresponds to  selfadjoint elements of $\cA_\theta$. Let $\cA_\theta^\R$ the real subspace of selfadjoint elements of $\cA_\theta$. More generally, for any $m\geq 1$, we denote by $M_n(\cA_\theta^\R)$ the real subspace of $M_m(\cA_\theta)$ of matrices with selfadjoint entries. Note that Hermitian elements of $M_n(\cA_\theta^\R)$ are symmetric matrices. 

 In general the inverse of an element of $M_m(\cA_\theta^\R)\cap \GL_m(\cA_\theta)$ need not have selfadjoint entries (see~\cite{HP:Laplacian}). In what follows we denote by $\GL_m^+(\cA_\theta^\R)$ the set of  matrices $g\in\GL^+_m(\cA_\theta)$ such that $g$ and its inverse $g^{-1}$ have selfadjoint entries. 

\begin{definition}[\cite{HP:Laplacian, Ro:SIGMA13}] A \emph{Riemmannian metric} on $\T^n_\theta$ is a Hermitian metric on $\sX_\theta$ whose matrix $g=(g_{ij})$ is in $\GL^+_n(\cA_\theta^\R)$. 
\end{definition}

\begin{remark}
 As usual we will often identify Riemannian metrics and their matrices. 
\end{remark}

\begin{example}\label{ex:Riemannian.conf-flat} 
 The standard flat metric is $g_{ij}=\delta_{ij}$. A conformal deformation is of the form $g_{ij}=k^2 \delta_{ij}$ with $k\in \cA_\theta^{++}$. This is the kind of Riemannian metric considered in~\cite{CM:JAMS14, CT:Baltimore11}.  
\end{example}

\begin{example}
 A product metric is of the form, 
\begin{equation}
 g = 
\begin{pmatrix}
 g_1 & 0\\
 0 & g_2
\end{pmatrix} = 
 \begin{pmatrix}
 g_1 & 0\\
 0 & 1
\end{pmatrix} \begin{pmatrix}
 1 & 0\\
 0 & g_2
\end{pmatrix}, \qquad g_j\in \GL_{n_j}^+(\cA_\theta^\R), \quad n_1+n_2=n.  
\label{eq:Riemannian.product-metric}
\end{equation}
 This includes the kind of Riemannian metrics considered in~\cite{CF:MJM19, DS:SIGMA15, DGK:arXiv18}. 
\end{example}

In what follows we shall say that matrices $a=(a_{ij})\in M_m(\cA_\theta)$ and $b=(b_{kl})\in M_{m'}(\cA_\theta)$ are \emph{compatible} when $[a_{ij},b_{kl}]=0$ for all $i,j=1,\ldots,m$ and $k,l=1,\ldots, m'$. We say that $a$ is \emph{self-compatible} when it is compatible with itself.

\begin{example}[Self-compatible Riemannian metrics~\cite{HP:Laplacian}] 
 Let $g=(g_{ij})\in \GL_n^{+}(\cA_\theta)$ be self-compatible. In this case the inverse $g^{-1}$ has selfadjoint entries as well, and so $g\in GL_n^{+}(\cA_\theta^\R)$, i.e., $g$ is a Riemannian metric. This example includes the conformally flat metrics of Example~\ref{ex:Riemannian.conf-flat}. It also includes the \emph{functional metrics} of~\cite{GK:arXiv18}.  
\end{example}

\subsection{Riemannian density} 
On a Riemannian manifold $(M^n,g)$ the Riemannian density is given in local coordinates by integration against $\sqrt{\det (g(x))}$. The volume $(M,g)$ is then obtained as the integral of the Riemannian density. More generally, smooth densities on $M$ are given by integration against positive functions in local coordinates. 

On the NC torus $\T_\theta^n$ the role of positive functions is played by elements of $\cA_\theta^{++}$. It is natural to think of any $\nu \in \cA_\theta^{++}$ as a smooth density on $T_\theta^n$, or if we think in terms of measures, to think as density the corresponding weight $\varphi_{\nu}(a)=\tau[a\nu]$, $a \in A_\theta$ (see~\cite{HP:Laplacian}). 

In order to define Riemannian densities on $\T_\theta^n$ we need a notion of determinant. Recall that if  $h\in \GL_n(\C)$ is positive-definite, then $\det(h) = \exp[ \Tr (\log h)]$. 

\begin{definition}[\cite{FK:AM52, HP:Laplacian}]  
The \emph{determinant} of any  $h\in \GL_n^+(\cA_\theta)$ is defined by 
\begin{equation*}
 \det (h) = \exp \big[ \Tr (\log h) \big], \qquad h\in \GL_n^+(\cA_\theta)\in \cA_\theta^{++},
\end{equation*}
where $\log h$ is defined by holomorphic functional calculus and $\Tr:M_n(\cA)\rightarrow \cA_\theta$ is the matrix trace.
\end{definition}

\begin{remark}
 If $h\in \GL_n^+(\cA_\theta)$, then $\log h$ is a selfadjoint element of $M_n(\cA_\theta)$, and so $\Tr (\log h)$ is a selfadjoint element of $\cA_\theta$. Therefore,  its exponential is selfadjoint and has positive spectrum, i.e., it is contained in $\cA_\theta^{++}$. 
\end{remark}

\begin{proposition}[\cite{HP:Laplacian}]\label{prop:Laplace.det-properties}
 Let $h=(h_{ij})$ be in $\GL_n^+(\cA_\theta)$. 
\begin{enumerate}
 \item[(i)] $\det(h^s)=(\det h)^s$ for all $s\in \R$. 
 
 \item[(ii)] Let $h'\in \GL_n^+(\cA_\theta)$ be compatible and commute with $h$. Then 
\begin{equation*}
  \det(hh')=\det(h)\det(h')=\det(h')\det(h). 
\end{equation*}
\item[(iii)] If $h$ is self-compatible, then we have 
 \begin{equation}
 \det (h)= \sum_{\sigma \in \fS_m} \varepsilon(\sigma) h_{1\sigma(1)} \cdots h_{m \sigma(m)}. 
 \label{eq:det.Leibniz}
\end{equation}
\end{enumerate}
\end{proposition}

We are now in a position to define Riemannian densities. 
 
 \begin{definition}\label{def:Riemannian.volume}
 Let $g=(g_{ij})\in \GL_n^+(\cA_\theta^\R)$ be a Riemannian metric. 
\begin{itemize}
\item The \emph{Riemannian density} is $\nu(g):=\sqrt{\det (g)}=\exp\big[\frac12 \Tr\left(\log (g)\right)\big]\in \cA_\theta^{++}$. 

\item The \emph{Riemannian weight} $\varphi_g: A_\theta \rightarrow \C$ is defined by
\begin{equation*}
 \varphi_g(a):=(2\pi)^n\tau\big[a\nu(g)\big], \qquad a \in A_\theta. 
\end{equation*}

\item The \emph{Riemannian volume} is $\Vol_g(\T^n_\theta) := \varphi_g(1) =(2\pi)^n\tau\big[\nu(g)\big]$. 
\end{itemize}
\end{definition}

\begin{example}
 When $g$ is the standard flat metric $g_{ij}=\delta_{ij}$, we have $\Vol_g(\T^n_\theta)=(2\pi)^{n}=|\T^n|$.
\end{example}

\begin{example}
 Let $g_{ij}=k^2\delta_{ij}$, $k\in \cA_\theta^{++}$, be a conformal deformation of the flat Eucldean metric. Then we have
 \begin{equation*}
 \nu(g)=k^n, \qquad \varphi_g(a)= (2\pi)^n\tau\big[ak^n\big], \qquad \Vol_g(\T^n_\theta)=(2\pi)^n\tau\big[k^n\big]. 
\end{equation*}
\end{example}

\begin{example}
 Let $g$ be a product matrix as in~(\ref{eq:Riemannian.product-metric}), where $g_1$ and $g_2$ are compatible. Then
\begin{equation*}
 \nu(g)=\nu(g_1)\nu(g_2). 
\end{equation*}
\end{example}

The very datum of the Riemannian density gives rise to a non-trivial modular geometry. Namely, unless $\nu(g)$ is a scalar the Riemannian weight $\varphi_g$ is not a trace. Therefore, the GNS construction gives rise to non-isometric $*$-representations of the $C^*$-algebra $A_\theta$ and its opposite algebra $A_\theta^\circ$. The latter is represented in the Hilbert space $\sH_g^\circ$ obtained as the completion with respect to the inner product,
\begin{equation*}
 \scal{u}{v}_g^\circ =\varphi_g\big(uv^*\big)= (2\pi)^n\tau\big[v^*\nu(g)u\big], \qquad u,v\in \cA_\theta.
\end{equation*}
The Tomita involution and modular operator are $J_g(a)=\sigma_g(a^*)$ and $\mathbf{\Delta}(a):=\sigma_\nu^2(a)$, where $\sigma_g$ is the inner automorphism, 
\begin{equation*}
 \sigma_g(a)=\nu(g)^{\frac12} a \nu(g)^{-\frac12}, \qquad a \in A_\theta. 
\end{equation*}

\subsection{Laplace-Beltrami operator} 
Let $g=(g_{ij})\in \GL_n^+(\cA_\theta^\R)$ be a Riemannian metric on $\T^n_\theta$.  

\begin{definition}[\cite{HP:Laplacian}] The \emph{Laplace-Beltrami operator} $\Delta_g:\cA_\theta \rightarrow \cA_\theta$ is defined by
\begin{equation}
  \Delta_{g}u  =  \nu(g)^{-1} \sum_{1\leq i,j \leq n} \delta_i \big( \sqrt{\nu(g)} g^{ij} \sqrt{\nu(g)} \delta_j(u)\big), \qquad u \in \cA_\theta.  
  \label{eq:Riemannian.Laplace-Beltrami} 
\end{equation}
 \end{definition}
\begin{example}
 For the flat metric $g_{ij}=\delta_{ij}$ the Laplace-Beltrami operator is just the flat Laplacian $\Delta=\delta_1^2 + \cdots + \delta_n^2$. 
\end{example}

\begin{example}
 Let $g=k^2\delta_{ij}$, $k\in \cA_\theta^{++}$, be a conformally flat metric. In this case, we have
\begin{equation}
 \Delta_g=k^{-2} \Delta + \sum_{1\leq i \leq n} k^{-n}  \delta_i(k^{n-2}) \delta_i. 
 \label{eq:Riemannian.Laplacian-conformally-flat}
\end{equation}
In particular, in dimension $n=2$ we get $\Delta_g=k^{-2}\Delta$. This is reminiscent of the conformal invariance of the Laplace-Beltrami operator in dimension~2. 
\end{example}

In what follows, we set 
\begin{equation}
 |\xi|_g:=\big(\sum_{i,j} g^{ij} \xi_i\xi_j\big)^{\frac12}, \qquad \xi\in \R^n. 
\label{eq:Laplace.norm-g}
\end{equation}
Note that $|\xi|_g\in \cA_\theta^{++}$ when $\xi\neq 0$. In fact (see, e.g., \cite[Corollary~3.1]{HP:Laplacian}) there even are constants $0<d_1<d_2$ such that
\begin{equation}
 d_1|\xi|^2 \leq |\xi|_g^2 \leq d_2|\xi|^2 \qquad \forall \xi\in \R^n. 
 \label{eq:Laplacian.positivity-normg}
\end{equation}

\begin{proposition}[\cite{HP:Laplacian}]\label{prop:Laplacian.properties}
 The following holds.
\begin{enumerate}
 \item $\Delta_g$ is an elliptic 2nd order differential operator with principal symbol $\nu(g)^{-\frac12}|\xi|_g\nu(g)^{\frac12}$. 
 
 \item $\Delta_g$ is an essentially selfadjoint operator on $\sH_g^\circ$ with non-negative discrete spectrum with finite multiplicity. 
 
 \item $\ker \Delta_g=\C$. 
\end{enumerate}
\end{proposition}

\section{ConnesÕ integration formula for Curved NC Tori}\label{sec:Curved-Integration}
In this section, we extend the Connes integration formula of~\cite{MSZ:MA19} to curved NC tori and present a few consequences.

\subsection{Spectral Riemannian densities}
For every positive-definite matrix $g \in \GL_n(\R)$, we have 
\begin{equation}
 \int_{\bS^{n-1}} |\xi|_g^{-n} d^{n-1}\xi = \sqrt{\det (g)} |\bS^{n-1}|, 
 \label{eq:Curved.norm-g-volume} 
\end{equation}
where $|\xi|_g$ is defined as in~(\ref{eq:Laplace.norm-g}). Indeed, as $ \sqrt{\det (g)} |\bS^{n-1}|$ is the volume of the ellipsoid $|\xi|^2_g=1$, we get
\begin{equation}
 \int_{\R^n} e^{-|\xi|_g} d\xi= \bigg( \int_0^\infty e^{-t^2} t^{n-1}dt 
 \bigg) \sqrt{\det (g)} |\bS^{n-1}|. 
 \label{eq:Curved.Gaussian-normg-volume} 
\end{equation}
Integrating in polar coordinates also shows that $ \int_{\R^n} e^{-|\xi|_g} d\xi$ is equal to 
\begin{equation}
 \int_{\bS^{n-1}} \bigg( \int^\infty_0 e^{-t^2|\xi|_g^2} t^{n-1}dt\bigg) d^{n-1}\xi =  \bigg( \int_0^\infty e^{-t^2} t^{n-1}dt 
 \bigg) \bigg( \int_{\bS^{n-1}} |\xi|_g^{-n} d^{n-1}\xi\bigg).
 \label{eq:Curved.Gaussian-normg-polar} 
\end{equation}
Comparing this to~(\ref{eq:Curved.Gaussian-normg-volume}) gives~(\ref{eq:Curved.norm-g-volume}). Note that combining Connes' trace theorem with
with~(\ref{eq:Curved.norm-g-volume}) yields the integration formula~(\ref{eq:Quantized.integration-formula}).

\begin{definition}
 Let $g\in \GL_n^+(\cA_\theta^\R)$ be a Riemannian metric on $\T^n_\theta$. Its \emph{spectral Riemannian density} is 
\begin{equation*}
 \tilde{\nu}(g) := \frac{1}{|\bS^{n-1}|} \int_{\bS^{n-1}} |\xi|_g^{-n} d^{n-1}\xi,
\end{equation*}
where $|\xi|_g$ is defined in~(\ref{eq:Laplace.norm-g}). 
\end{definition}

The following statement gathers the main properties of spectral Riemannian densities. 

\begin{proposition}\label{prop:Curved.properties-tnug}
Let $g\in \GL_n^+(\cA_\theta^\R)$ be a Riemannian metric on $\T^n_\theta$. 
\begin{enumerate}
 \item $\tilde{\nu}(g)\in \cA_\theta^{++}$, i.e., $\tilde{\nu}(g)$ is a smooth density. 
 
 \item We have \[\tilde{\nu}(g)= \pi^{-\frac{n}2}\int_{\R^n} e^{-|\xi|_g^2} d\xi.\] 
 
 \item Let $k\in \cA_\theta^{++}$ be such that $[k,g]=0$. Then $\tilde{\nu}(k^2g)=k^n \tilde{\nu}(g)$. 

\item If $g$ is self-compatible, then 
\begin{equation}\label{eq:Curved.tnug-compatible}
 \tilde{\nu}(g)=\nu(g)=\sqrt{\det (g)},
\end{equation}
where $\det(g)$ is given by~(\ref{eq:det.Leibniz}). 
\end{enumerate}
\end{proposition}
\begin{proof}
 As $\xi \rightarrow |\xi|_g^{-n}$ is a continuous map with values in the closed real subspace $\cA_\theta^\R$, it is immediate that $\tilde{\nu}(g)\in \cA_\theta^\R$. 
 Therefore, in order to check that $\tilde{\nu}(g)\in \cA_\theta^{++}$ we only need to show that $\tilde{\nu}(g)$ has positive spectrum. 

If $\xi\in \bS^{n-1}$, then~(\ref{eq:Laplacian.positivity-normg}) implies that $\Sp(|\xi|_g)\subset [d_1,d_2]$, and so $\Sp(|\xi|_g^{-n})\subset [d_2',d_1']$, where $d_i':=d_i^{-\frac{n}{2}}>0$. Thus, given any $u\in \cH_\theta$, we have $\scal{|\xi|_g^{-n}u}{u}\geq d_2' \scal{u}{u}$, and so we have 
\begin{equation*}
 \scal{|\xi|_g^{-n}u}{u} = |\bS^{n-1}|^{-1} \int_{\bS^{n-1}} \scal{|\xi|_g^{-n}u}{u} d^{n-1}\xi \geq d_2' \scal{u}{u}. 
\end{equation*}
This implies that $\Sp(\tilde{\nu}(g))\subset [d_2',\infty)\subset (0,\infty)$, and so this gives the first part. 

 In the same way as in~(\ref{eq:Curved.Gaussian-normg-polar}) we see that $ \int_{\R^n} e^{-|\xi|_g^2} d\xi$ is equal to 
\begin{equation*}
 \bigg( \int_0^\infty e^{-t^2} t^{n-1}dt 
 \bigg) \bigg( \int_{\bS^{n-1}} |\xi|_g^{-n} d^{n-1}\xi\bigg)=  \bigg( \int_0^\infty t^{n-1}e^{-t^2} dt 
 \bigg)|\bS^{n-1}|\tilde{\nu}(g). 
\end{equation*}
Here $|\bS^{n-1}|= 2\pi^{\frac{n}2} \Gamma(n/2)^{-1}$, and $ \int_0^\infty t^{n-1}e^{-t^2}dt$ is equal to
\begin{equation*}
 \int_0^\infty t^{\frac{n-1}{2}}e^{-t}\frac{dt}{2\sqrt{t}}  =  \frac{1}{2}\int_0^\infty t^{\frac{n}{2}-1}e^{-t}dt=\frac{1}{2}\Gamma\big(\frac{n}{2}\big)= \pi^{\frac{n}2}|\bS^{n-1}|^{-1}. 
\end{equation*}
It then follows that $\int_{\R^n} e^{-|\xi|_g^2} d\xi=\pi^{\frac{n}2}\tilde{\nu}(g)$. This proves the 2nd part. 

Let $k\in \cA_\theta^{++}$ be such that $[k,g]=0$, and set $\hat{g}=k^2g$. As $k$ commutes with $g^{-1}$, we have $\hat{g}^{-1}=k^{-2}g^{-1}$, and so, for all $\xi\in \R^n$, we have $|\xi|^2_{\hat{g}}=\sum k^{-2}g^{ij}\xi_i\xi_j=k^{-2}|\xi|_{g}^2$ and $|\xi|_{\hat{g}}^{-n}=k^{n}|\xi|_{g}^{-n}$. Thus, 
\[
 \tilde{\nu}(\hat{g})= \int_{\bS^{n-1}} |\xi|_{\hat{g}}^{-n}d\sigma_0(\xi)=    \int_{\bS^{n-1}} k^{n}|\xi|_{g}^{-n}d\sigma_0(\xi)= k^n \tilde{\nu}(g).  
 \] 
 This gives the 3rd part. 
 
 It remains to prove the 4th part. We only have to show that the formula~(\ref{eq:Curved.norm-g-volume}) continues to hold for self-compatible matrices in $\GL_n^+(\cA_\theta^\R)$. Note that this formula immediately extends to any $g\in C(X,\GL_n^+(\R))$, where $X$ is any compact Hausdorff space. 

Suppose that $g$ is self-compatible. The entries $g_{ij}$ of $g$ pairwise commute with each other. Thus, if we let $B$ be the unital $C^*$-algebra that they generate, then we obtain a unital commutative $C^*$-algebra (\emph{cf}.~\cite{HP:Laplacian}). Therefore, under the Gel'fand transform the $C^*$-algebra $B$ is isomorphic to the $C^*$-algebra of continuous functions on the Gel'fand spectrum $X:=\Sp(B)$. As~(\ref{eq:Curved.norm-g-volume}) holds for any element of  $C(X,\GL_n^+(\R))$ we deduce that it holds for $g$. This proves the 4th part and completes the proof.  
\end{proof}

\begin{example}
 Let $g$ be of the form, 
 \begin{equation*}
 g= \begin{pmatrix}
 k_1^2 I_{n_1} & 0 \\
 0 & k_2^2 I_{n_2} 
\end{pmatrix}, \qquad k_j\in \cA_\theta^{++}, \ n_1+n_2=n.
\end{equation*}
Assume further that $[k_1,k_2]=0$. This ensures that $g$ is self-compatible, and so by combining~(\ref{eq:Curved.tnug-compatible}) with Proposition~\ref{prop:Laplace.det-properties} we get
\begin{equation*}
 \tilde{\nu}(g)=\sqrt{\det(g)}=\sqrt{\det(k_1I_1)\det(k_2I_n)}=\sqrt{(k_1k_2)^{2n}}=(k_1k_2)^n. 
\end{equation*}
\end{example}

\subsection{Curved integration formula} Let $g\in \GL_n^+(\cA_\theta^\R)$ be a Riemannian metric on $\T^n_\theta$. As the Hilbert space $\sH_g^\circ$ arises from the GNS construction for the opposite $C^*$-algebra $A_\theta^\circ$, the right action of $\cA_\theta$ on itself by right-multiplication extends to a $*$-representation in $\cH_g^\circ$ of the opposite $C^*$-algebra $A_\theta^\circ$. The left-action by left-multiplication extends to a representation of $A_\theta$. This not $*$-representation. We obtain a $*$-representation by using the Tomita involution $J_g(a)=\sigma_{\nu(g)}(a^*)=\nu(g)^{-\frac12} a^* \nu(g)^{\frac12}$. Namely, we get the $*$-representation, 
\begin{equation*}
 a \longrightarrow a^\circ, \qquad \text{where}\ a^\circ:=J_ga^*J_g = \nu(g)^{-\frac12} a \nu(g)^{\frac12}. 
\end{equation*}
An intertwining unitary isomorphism from $\cH_g^\circ$ onto $\cH_\theta$ is the left-multiplication by $(2\pi)^{\frac{n}2}\sqrt{\nu(g)}$.

Let $\Delta_g:\cH_\theta \rightarrow \cH_\theta$ be the Laplace-Beltrami operator defined by the Riemannian metric $g$. By Proposition~\ref{prop:Laplacian.properties} it is essentially selfadjoint on $\cH_g^\circ$ and has non-negative spectrum. Under the unitary isomorphism $(2\pi)^{\frac{n}2}\nu(g)^{\frac12}:\cH_g^\circ \rightarrow \cH_\theta$ it corresponds to the operator, 
\begin{equation}
 \cDelta_g:= \nu(g)^{\frac12} \Delta_g \nu(g)^{-\frac12}= \sum_{1\leq i,j\leq n} \nu(g)^{-\frac12} \delta_i \nu(g)^{\frac12}g^{ij}\nu(g)^{\frac12}\delta_j\nu(g)^{-\frac12}.
 \label{eq:Curved.Laplacian-prime}  
\end{equation}
This operator is essentially selfadjoint with non-negative spectrum on $\cH_\theta$. It is also a 2nd order differential operator with principal symbol $ \sum_{i,j} \xi_i g^{ij}\xi_j=|\xi|_g^2$. All this allows us to define for every $s>0$ the operators $\Delta_g^{-s}$ and $\cDelta_g^{-s}$ by standard Borel functional calculus on $\cH_g^\circ$ and $\cH_\theta$, respectively.  Note that $\Delta_g^{-s}=\nu(g)^{-\frac12}\cDelta_g^{-s} \nu(g)^{\frac12}$. 

\begin{lemma}\label{lem:Curved.symbol-Deltagn}
 For every $s>0$, the operators $\cDelta_g^{-s}$ and $\Delta_g^{-s}$ are in $\Psi^{-2s}(\T_\theta^n)$ and have principal symbols $|\xi|_g^{-2s}$ and $\nu(g)^{-\frac12} |\xi|_g^{-2s} \nu(g)^{\frac12}$, respectively. 
\end{lemma}
\begin{proof}
 As mentioned above, $\cDelta_g$ is a 2nd order elliptic differential operator. It is essentially selfadjoint with non-negative spectrum on $\cH_\theta$ and its principal symbol $|\xi|_g$ is positive and invertible for $\xi\neq 0$. Therefore, it satisfies the assumptions of~\cite[Theorem~7.3]{LP:Resolvent}, and hence $\cDelta_g^{-s}$ is an operator in $\Psi^{2s}(\T_\theta^n)$ with principal symbol $|\xi|_g^{-2s}$. As $\Delta_g^{-s}=\nu(g)^{-\frac12}\cDelta_g^{-s} \nu(g)^{\frac12}$, it then follows from Corollary~\ref{cor:PsiDOs.composition-classical} that $\Delta_g^{-s}$ is in $\Psi^{2s}(\T_\theta^n)$ and has principal symbol $\nu(g)^{-\frac12} |\xi|_g^{-2s} \nu(g)^{\frac12}$. The proof is complete. 
\end{proof}

\begin{remark}
The above result actually holds for any $s\in \C$ (see~\cite[\S7]{LP:Resolvent}). When $s\in \Z_{-}$ the result is immediate, since $\cDelta_g^{-s}$ and $\Delta_g^{-s}$ then are positive-integer powers of differential operators. When $s\in \N$ this also can be deduced from the case $s\in \Z_{-}$ and the results of~\cite[\S11]{HLP:Part2}, since $\Delta_g^{-s}$ and $\cDelta_g^{-s}$ are parametrices of the differential operators $\Delta_g^{s}$ and $\cDelta_g^{s}$, respectively. 
\end{remark}

\begin{remark}
In the case of a conformally flat metric $g_{ij}=k^2\delta_{ij}$, $k\in \cA_\theta^{++}$,  by using~(\ref{eq:Riemannian.Laplacian-conformally-flat}) we get
\begin{align*}
 \cDelta_g & = k^{\frac{n}2}\big(k^{-2} \Delta + \sum_{1\leq i \leq n} k^{-n}  \delta_i(k^{n-2}) \delta_i\big)k^{-\frac{n}2}\\ 
 & = k^{\frac{n}2-2}\Delta k^{-\frac{n}2} + \sum_{1\leq i \leq n} k^{-\frac{n}2}\big(\delta_i(k^{n-2}) \delta_i\big)k^{-\frac{n}2}. 
\end{align*}
In particular, when $n=2$ we recover the conformally perturbed Laplacian $k^{-1}\Delta k^{-1}$ of~\cite{CT:Baltimore11}. 
\end{remark}

We are now in a position to  prove the main result of this section. 

\begin{theorem}[Curved Integration Formula]\label{thm:Curved.integration-formula}
For every $a\in A_\theta$, the operator $a^\circ \Delta_g^{-\frac{n}2}$ is strongly measurable, and we have
\begin{equation}
 \bint a^\circ \Delta_g^{-\frac{n}2} = \bcn \tau\big[a\tilde{\nu}(g)\big], \qquad \bcn:=\frac{1}{n}|\bS^{n-1}|.
 \label{eq:curved.curved-formula} 
\end{equation}
\end{theorem}
\begin{proof}
 We know by Lemma~\ref{lem:Curved.symbol-Deltagn} that $\Delta_g^{-\frac{n}2}$ and $\cDelta_g^{-\frac{n}2}$ are operators in $\Psi^{-n}(\T_\theta^n)$.  Let $a\in A_\theta$. For every positive trace $\varphi$ on $\cL^{1,\infty}$, we have
 \begin{equation}
 \varphi\big(a^\circ \Delta_g^{-\frac{n}2}\big)= \varphi\left[\big(\nu(g)^{-\frac12}a \nu(g)^{\frac12}\big)\Delta_g^{-\frac{n}2}\right] =
  \varphi\left[a   \big(\nu(g)^{\frac12}\big)\Delta_g^{-\frac{n}2}\nu(g)^{-\frac12}\big)\right]=\varphi\big(a \cDelta_g^{-\frac{n}2}\big). 
  \label{eq:curved.trace}
\end{equation}
 Combining this with Corollary~\ref{cor:Trace-thm.super-integration-formula} we deduce that $a^\circ \Delta_g^{-\frac{n}2}$ is strongly measurable, and we have
\begin{equation}
 \bint a^\circ \Delta_g^{-\frac{n}2} = \bint a \cDelta_g^{-\frac{n}2} = \frac{1}{n}\tau\big[ac_{\cDelta_g^{-\frac{n}2}}\big]. 
 \label{eq:curved.binta} 
\end{equation}
 By Lemma~\ref{lem:Curved.symbol-Deltagn} the principal symbol of $\cDelta_g^{-\frac{n}2}$ is $|\xi|_g^{-n}$,  and so we have 
\begin{equation*}
 c_{\cDelta_g^{-\frac{n}2}} =\int_{\bS^{n-1}} |\xi|_g^{-n}d^{n-1}\xi = |\bS^{n-1}|\tilde{\nu}(g)=n\bcn\tilde{\nu}(g). 
\end{equation*}
 Combining this with~(\ref{eq:curved.binta}) completes the proof. 
\end{proof}

\begin{remark}
In the definition~(\ref{eq:Riemannian.Laplace-Beltrami}) of the Laplace-Beltrami operator we may replace $\nu(g)$ by any smooth density $\nu\in \cA_\theta^{++}$. The resulting operator, denoted by $\Delta_{g,\nu}$ in~\cite{HP:Laplacian}, is essentially selfadjoint on the Hilbert space $\sH_\nu^\circ$ of the GNS representation of $A_\theta^\circ$ with respect to the weight $\varphi_\nu(a)=\tau[a\nu]$, $a\in A_\theta$ (see~\cite{HP:Laplacian}). The $*$-representation $a\rightarrow a^\circ$  of $A_\theta$ in $\sH_\nu^\circ$ is given by 
 $a^\circ=\nu^{-\frac12}a\nu^{\frac12}$. With these modifications we have 
\begin{equation*}
 \bint a^\circ \Delta_{g, \nu}^{-\frac{n}2} = \bcn \tau\big[ a \tilde{\nu}(g)\big]  \qquad \forall a \in A_\theta. 
\end{equation*}
\end{remark}

By Proposition~\ref{prop:Curved.properties-tnug} the spectral Riemannian density $\tilde{\nu}(g)$ agrees with the Riemannian density $\nu(g)$ when $g$ is  self-compatible. Therefore, in this case we obtain the following result. 

\begin{corollary}\label{cor:Curved.integration-formula-selfcompatible}
Assume that $g$ is a self-compatible Riemannian metric. Then, for all $a\in A_\theta$, we have
 \begin{equation}
 \bint a^\circ \Delta_g^{-\frac{n}2} = \bcn \tau\big[a\nu(g)\big] =c_n\varphi_g(a), 
 \label{eq:Curved.integration-formula-selfcompatible} 
\end{equation}
where $c_n=(2\pi)^{-n}\bcn$ is the same constant as in~(\ref{eq:Quantized.integration-formula}). In particular, for $a=1$ we get 
\begin{equation}
 \bint \Delta_g^{-\frac{n}2}=c_n\Vol_g(\T^n_\theta).
 \label{eq:Curved.selfcompatible-volume}  
\end{equation}
\end{corollary}
 
\begin{remark}
 When $\theta=0$ every Riemannian metric is self-compatible. We then recover the integration formula~(\ref{eq:Quantized.integration-formula}) for arbitrary Riemannian metrics on the ordinary torus $\T^n$. 
\end{remark}

\begin{remark}
 The formula~(\ref{eq:Curved.integration-formula-selfcompatible}) shows that, in the self-compatible case, the NC integral allows us to recover the Riemannian density $\nu(g)$ and the Riemannian weight $\varphi_g$. As mentioned in the introduction this provides us with a nice link between the notions of integrals in NCG and the noncommutative measure theory in operator algebras, where the role of Radon measures is played by weights on $C^*$-algebras.  
 \end{remark}
 
\begin{remark}\label{rmk:curved.uniqueness}
 In order to recover the Riemannian  density $\nu(g)$ it is enough to have~(\ref{eq:Curved.integration-formula-selfcompatible}) for elements $a$ of the smooth algebra $\cA_\theta$. Indeed, by using the Fourier series 
 expansion $\nu(g)=\sum \nu_k(g)U^k$ in $\cA_\theta$, we get
\begin{equation*}
  \nu_k(g) =\scal{\nu(g)}{U^k}= \tau\big[(U^k)^*\nu(g)\big]=\bcn^{-1} \bint \big[(U^k)^{-1}]^{\circ} \Delta_g^{-n}. 
\end{equation*}
This confirms our claim. 
 \end{remark}

\subsection{Lower dimensional volumes} 
In the same way as in the case of ordinary manifolds, the trace formula~(\ref{eq:Trace-Thm.trace-formula}) allows us to extend the NC integral to all operators in $\Psi^\Z(\T_\theta^n)$, including those operators that are not weak trace class or even bounded. 

\begin{definition}
For every $P\in \Psi^m(\cA_\theta)$, $m\in \Z$, its NC integral is defined by
\begin{equation*}
 \bint P : = \frac{1}{n} \Res (P). 
\end{equation*} 
\end{definition}

When the metric $g$ is self-compatible the integration formula~(\ref{eq:Curved.integration-formula-selfcompatible}) shows that the NC integral recaptures the Riemannian density. For general Riemannian metrics, it recaptures the density $\tilde{\nu}(g)$, and so it is natural to call the latter the \emph{spectral Riemannian density} of $(\T^n_\theta, g)$.

Along the same lines of thought as that of Section~\ref{sec:Quantized}, we can regard the operator $ c_n^{-1} \Delta_g^{-\frac{n}2}$ as the \emph{quantum volume element}, and then we define the \emph{quantum length element} as the operator, 
\begin{equation*}
 ds:=\big(c_n^{-1} \Delta_g^{-\frac{n}2}\big)^{\frac{1}{n}}=c_n^{-\frac1{n}} \Delta_g^{-\frac{1}2}. 
\end{equation*}
Note that Lemma~\ref{lem:Curved.symbol-Deltagn} ensures us that $ds$ is an operator in $\Psi^{-1}(\T^n_\theta)$, and hence is an infinitesimal operator of order~$\frac{1}{n}$ by Proposition~\ref{prop:continuity-weak-Schatten-class}. 
In particular, we can define the NC integrals of the powers $ds^k$, $k=1,\ldots,n$. This leads us to the following definition of the lower dimensional \emph{spectral} volumes of curved NC tori. 

\begin{definition}
 For $k=1,\ldots, n$, the \emph{$k$-dimensional spectral volume} of $(\T^n_\theta, g)$ is
\begin{equation*}
 \widetilde{\Vol}^{(k)}_g(\T^n_\theta):=\bint ds^k. 
\end{equation*}
\end{definition}

\begin{remark}
 When the Riemannian metric $g$ is self-compatible, by using~(\ref{eq:Curved.selfcompatible-volume}) we get
\begin{equation*}
  \widetilde{\Vol}^{(n)}_g(\T^n_\theta)=\bint ds^n = c_n^{-1}\bint \Delta_g^{-\frac{n}{2}} = \Vol_g(\T^n_\theta). 
\end{equation*}
Thus, in this case the spectral volume $ \widetilde{\Vol}^{(n)}_g(\T^n_\theta)$ agrees with the Riemannian volume in the sense of Definition~\ref{def:Riemannian.volume}. 
\end{remark}

\section{Scalar Curvature of Curved NC Tori}\label{sec:curvature}  
In this section, we explain how the results of the previous section enables us to setup a natural notion of scalar curvature for general curved NC tori.  

As mentioned in Section~\ref{sec:Quantized}, for any (compact) Riemannian manifold $(M^n,g)$, the functional $C^\infty(M)\ni a\rightarrow -\frac{1}{c_n'}\bint a ds^{n-2}$ recaptures the (dressed) scalar curvature $\kappa(g)\nu(g)$. For curved NC tori $(\T^n_\theta,g)$ this leads us to the following definition. 

\begin{definition}\label{def:curvature.bint}
Assume $n\geq 3$, and let $g\in \GL_n(\cA_\theta^\R)$ be a Riemannian metric on $\T_\theta^n$. The \emph{scalar curvature functional} of $(\T^n_\theta,g)$ is the functional $\cR_g:\cA_\theta \rightarrow \C$ given by
\begin{equation*}
 \cR_g(a)=-\frac{1}{c_n'} \bint  a^\circ ds^{n-2}, \qquad a\in \cA_\theta. 
\end{equation*}
\end{definition}
 
 \begin{remark}
This definition is consistent with the definition of the scalar curvature functional of spectral triples~\cite[Definition~1.147]{CM:AMS08} (see also~\cite[Remark~5.4]{CM:JAMS14}) and with the intrinsic definition of the modular curvature of conformally flat NC 2-tori~\cite[\S5.2]{CM:JAMS14} (see also~\cite{FK:JNCG15}). 
\end{remark}
 
 In what follows we assume $n\geq 3$ and let $g\in \GL_n(\cA_\theta^\R)$ be a Riemannian metric on $\T_\theta^n$. 
 
\begin{proposition}\label{prop:curved.kappag}
 There is a unique $\kappa(g)\in \cA_\theta$ such that 
 \begin{equation}
 \cR_g(a) = \tau\big[ a \kappa(g)\big] \qquad \forall a\in \cA_\theta. 
 \label{eq:curvature.functional-kappa} 
\end{equation}
Namely, we have
\begin{equation}
\kappa(g)=-3(4\pi)^{\frac{n}2} \Gamma\big(\frac{n}2-1\big)c_{\hat{\Delta}_g^{-\frac{n}2+1}}.
\label{eq:curvature.scalar-curvature}
\end{equation}
\end{proposition}
\begin{proof}
 Let $a \in \cA_\theta$. In view of the definitions of $ds$ and $\bint$, we have
\begin{equation*}
 \cR_g(a)=-\frac{1}{c_n'} \bint  a^\circ ds^{n-2} =  -\frac{1}{c_n'}c_n^{-\frac{n-2}{n}} \bint a^\circ \Delta_g^{-\frac{n}2+1} = 
 -\frac{1}{nc_n'}c_n^{-\frac{n-2}{n}}  \Res \big( a^\circ \Delta_g^{-\frac{n}2+1}\big). 
\end{equation*}
By arguing as in~(\ref{eq:curved.trace}) and using~(\ref{eq:NCR.local}) we get 
 \begin{equation*}
 \Res \big( a^\circ \Delta_g^{-\frac{n}2+1}\big)=   \Res \big( a \cDelta_g^{-\frac{n}2+1}\big)=  \tau\big[ a c_{\hat{\Delta}_g^{-\frac{n}2+1}}\big].  
\end{equation*}
Moreover, it follows from~(\ref{eq:quantized.cn'}) that $\frac{1}{nc_n'}c_n^{-\frac{n-2}{n}}=12(4\pi)^{\frac{n}2}\Gamma(\frac{n}2-1)$. 
Thus, 
\begin{equation*}
 \cR_g(a)=-3(4\pi)^{\frac{n}2}\Gamma\big(\frac{n}2-1\big) \tau\big[ a c_{\hat{\Delta}_g^{-\frac{n}2+1}}\big]. 
   \label{eq:curvature.bintadsn-2}
\end{equation*}
That is,~(\ref{eq:curvature.functional-kappa}) holds with $\kappa(g)$ given by~(\ref{eq:curvature.scalar-curvature}). 

In the same way as in Remark~\ref{rmk:curved.uniqueness}, the formula~(\ref{eq:curvature.functional-kappa}) uniquely determines the Fourier coefficients of $\nu(g)$, and so $\nu(g)$ is unique. The proof is complete. 
\end{proof}

\begin{definition}
$\nu(g)$  is  the \emph{scalar curvature} of $(\T_\theta^n, g)$. 
\end{definition}

\begin{remark}
 When $\theta=0$ we recover the usual notion of scalar curvature for Riemannian metrics on the ordinary torus $\T^n$.  
\end{remark}

\begin{remark}
 In dimension $n=4$ and for conformally flat metrics we recover the definition of the scalar curvature of conformally flat metrics in~\cite{Fa:JMP15, FK:JNCG15} with a different sign convention and up to the factor $6(4\pi)^{\frac{n}2}\Gamma(\frac{n}2-1)=6(4\pi)^2$. Note that the factor $\Gamma(\frac{n}2-1)$ in~(\ref{eq:curvature.bintadsn-2}) is essential to relate $\kappa(g)$ to the symbol of the resolvent and perform the analytic continuation at $n=2$ (see Proposition~\ref{prop:curved.curvature-sigma4} and Remark~\ref{sec:curvature.2D}). 
\end{remark}

The scalar curvature of conformally flat tori is often expressed in terms of the subleading coefficient in the short time heat trace asymptotics (see, e.g.,~\cite{CM:JAMS14, CT:Baltimore11}). This heat coefficient is given by multi-integrals involving the 2nd subleading symbol $\sigma_{-4}(\xi;\lambda)$ of the resolvent $(\cDelta_g-\lambda)^{-1}$. We shall now explain how to derive a similar expression from the formula~(\ref{eq:curvature.scalar-curvature}). 

We need to determine the symbol of degree~$-n$ of $\cDelta_g^{-\frac{n}2+1}$. Note that as $\cDelta_g^{-\frac{n}2+1}$ has order $-n+2$ its symbol of degree~$-n$ agrees with its 2nd subleading symbol. Let $s>0$. We have
\begin{equation*}
 \cDelta_g^{-s} = \frac{1}{2i\pi} \int_\Gamma \lambda^{-s} (\cDelta_g-\lambda)^{-1}d\lambda,
\end{equation*}
where the integral converges in $\cL(\cH_\theta)$ and $\Gamma$ is a downward oriented contour of the form $\Gamma =\partial \Lambda(r)$, with
$\Lambda(r):= \left\{\lambda \in \C^*; \ \Re \lambda\leq 0 \ \text{or} \ |\lambda|\leq r\right\}$ and $0<r<\inf\Sp(\cDelta_g))$ (see~\cite{LP:Resolvent}). It is shown in~\cite{LP:Resolvent} that $(\cDelta_g-\lambda)^{-1}$ belongs to a suitable class of \psidos\ with parameter on $\T_\theta^n$. In particular, it has a symbol $\sigma(\xi;\lambda)\sim \sum_{j\geq 0} \sigma_{-2-j}(\xi;\lambda)$, where
\begin{gather*}
 \sigma_{-2}(\xi;\lambda) = \big(|\xi|_g^2-\lambda \big)^{-1},\\
 \sigma_{-2-j}(\xi;\lambda) = - \sum_{\substack{k+l+|\alpha|=j\\ l<j}} \frac{1}{\alpha!}  \big(|\xi|_g^2-\lambda \big)^{-1} \partial_\xi^\alpha p_{2-k}(\xi)\delta^\alpha \sigma_{-2-l}(\xi;\lambda), \quad j\geq 0. 
\end{gather*}
Here $p(\xi)=p_2(\xi)+p_1(\xi)+p_0$ is the symbol of $\cDelta_g$ with $p_2(\xi)=|\xi|_g^2$. We have the following refinement of Lemma~\ref{lem:Curved.symbol-Deltagn}. 

\begin{lemma}[{\cite[Theorem~7.3]{LP:Resolvent}}] 
 Let $s<0$. The operator $\cDelta_g^{-s}$ has symbol $\rho(s;\xi)\sim \sum_{j\geq 0} \rho_{-2s-j}(s;\xi)$, where 
\begin{equation}
 \rho_{-2s-j}(s;\xi)= \frac{1}{2i\pi} \int_{\gamma_\xi} \lambda^{-s} \sigma_{-2-j}(\xi;\lambda)d\lambda, \qquad \xi\neq 0.
 \label{eq:scalar.symbol-powers}
\end{equation}
Here $\gamma_\xi$ is any clockwise-oriented Jordan curve in $\C\setminus (-\infty,0]$ that encloses $\Sp(|\xi|_g^2)$. 
\end{lemma}

\begin{remark}
 The above result actually holds for any $s\in \C$ (see~\cite{LP:Resolvent}). 
\end{remark}

For $s=\frac{n}2-1$ we obtain the symbol $\rho_{-n}(\frac{n}2-1;\xi)$ by setting $j=2$ in~(\ref{eq:scalar.symbol-powers}). Moreover, it  follows from~(\ref{eq:Laplacian.positivity-normg}) that $\Sp(|\xi|_g^2)\subset [d_1,d_2]$ with $0<d_1<d_2$. In particular, if $\xi\in \bS^{n-1}$, then $\Sp(|\xi|_g^2)\subset [d_1,d_2]$, and so may take $\gamma_\xi=\gamma$, where $\gamma$ is any given clockwise-oriented Jordan curve in $\C\setminus (-\infty,0]$ that encloses $[d_1,d_2]$. We then get
\begin{equation*}
c_{\cDelta_g^{-\frac{n}2+1}}= \int_{\bS^{n-1}} \rho_{-n}(\frac{n}2-1;\xi)d^{n-1}\xi = \frac{1}{2i\pi} \int_{\bS^{n-1}} \int_\gamma \lambda^{-\frac{n}2+1} \sigma_{-4}(\xi;\lambda)d^{n-1}\xi d\lambda .
\end{equation*}
Combining this with~(\ref{eq:curvature.scalar-curvature})  we then arrive at the following result. 

\begin{proposition}\label{prop:curved.curvature-sigma4} 
 Assume $n\geq 3$,  and let $g\in \GL_n(\cA_\theta^\R)$ be a Riemannian metric on $\T_\theta^n$. Then we have
\begin{equation}
 \kappa(g)=-\frac{3}{2i\pi} (4\pi)^{\frac{n}2} \Gamma\big(\frac{n}2-1\big) \int_{\bS^{n-1}} \int_\gamma \lambda^{-\frac{n}2+1} \sigma_{-4}(\xi;\lambda)d^{n-1}\xi d\lambda, 
 \label{eq:curvature.scalar-curvature-symbol}
\end{equation}
where $\sigma_{-4}(\xi;\lambda)$ and $\gamma$ are as above. 
\end{proposition}

\begin{remark}\label{rmk:curvature.heat}
For $\Re \lambda>0$ we have $ \Gamma(\frac{n}2-1)\lambda^{-\frac{n}2+1}=\int_0^\infty t^{\frac{n}2-2} e^{-t\lambda}dt= 2\int_0^\infty t^{n-3} e^{-t^2\lambda}dt$. Inserting this into the integral on the r.h.s.~of~(\ref{eq:curvature.scalar-curvature-symbol}) and using the homogeneity of $\sigma_{-4}(\xi;\lambda)$ we find that
\begin{equation}
  \kappa(g)=-\frac{6}{2i\pi} (4\pi)^{\frac{n}2}  \int_{\R^{n}} \bigg( \int_{\Gamma'} e^{-\lambda} \sigma_{-4}(\xi;\lambda) d\lambda \bigg)d\xi. 
  \label{eq:curvature.scalar-curvature-symbol-heat}
\end{equation}
where $\Gamma'$ is any upward-oriented contour $\{|\arg \lambda|=\phi\}$, $0<\phi <\frac{\pi}2$. This kind of formula for the scalar curvature is not new; it is used by various authors (see, e.g., \cite{CM:JAMS14, CT:Baltimore11, DGK:arXiv18}). The advantage of using the integral in~(\ref{eq:curvature.scalar-curvature-symbol}) with respect to that given in~(\ref{eq:curvature.scalar-curvature-symbol-heat}) is two-fold. First, we get an integration domain that is one-dimension smaller. Second, this domain is bounded, and so we avoid the convergence issues with integrating over unbounded domains in~(\ref{eq:curvature.scalar-curvature-symbol-heat}). 
\end{remark}
 
 \begin{remark}
 We observe from~\cite{LM:GAFA16, SZ:arXiv19} that  
 $(2i\pi)^{-1}\int_{\R^{n}} \big( \int_{\Gamma'} e^{-\lambda} \sigma_{-4}(\xi;\lambda) d\lambda \big)d\xi$ agrees with
the coefficient $a_2(\Delta_g)$ in the short time asymptotics, 
\begin{equation}
 \Tr\big[ b^\circ e^{-t\Delta_g}\big] \sim t^{-\frac{n}2}\sum_{j\geq 0} t^{j} \tau\big[ b a_{2j}(\Delta_g)\big], \qquad b\in \cA_\theta. 
\end{equation}
 Therefore, from~(\ref{eq:curvature.scalar-curvature-symbol-heat}) we also can infer that 
 \begin{equation*}
 a_2(\Delta_g)= - \frac{1}{6}(4\pi)^{-\frac{n}2}\kappa(g). 
\end{equation*}
In particular, when $\theta=0$ we recover the well-known formula for $a_2(\Delta_g)$ in terms of $\kappa(g)$ (see~\cite{Gi:CRC95}). This shows the consistency of our normalization constants. 
\end{remark}
 
\begin{remark}\label{sec:curvature.2D}
 We obtain from~(\ref{eq:curvature.scalar-curvature-symbol}) the scalar curvature of NC 2-tori by analytic continuation of $\Gamma(\frac{n}2-1)\lambda^{-\frac{n}2+1}$ at $n=2$. Namely, 
\begin{equation*}
  \kappa(g)=-\frac{12}{2i \pi} \int_{\bS^{1}} \int_\gamma \sigma_{-4}(\xi;\lambda)d\xi d\lambda .
\end{equation*}
This is consistent with the definition of the definition of the scalar curvature functional in dimension~2 (see~\cite[Eq.~(5.23)]{CM:JAMS14}). In particular,  in the same way as in Remark~\ref{rmk:curvature.heat}, for conformally flat NC 2-tori we recover the formulas for the scalar curvature in terms of $\sigma_{-4}(\xi;\lambda)$ of~\cite{CM:JAMS14, CT:Baltimore11}. 
\end{remark}

\begin{remark}
 Rosenberg~\cite{Ro:SIGMA13} introduced the analogue the Levi-Civita connection and Riemann curvature tensor for Riemannian metrics on NC tori. This leads us to an 
 alternative notion of scalar curvature for curved NC tori. 
\end{remark}

\appendix

\section{Proof of Equality~(\ref{eq:Quantized.scalar-curvature})}\label{sec:quantized-EH}
In this appendix we briefly sketch a proof of Eq.~(\ref{eq:Quantized.scalar-curvature}). Suppose that $(M^n,g)$ is a closed Riemannian manifold, and let $f\in C^\infty(M)$. In view of~(\ref{eq:quantized.extension-bint}) and of the fact that $ds=c_n^{-\frac{1}{n}}\Delta_g^{-\frac12}$ we have
\begin{equation}
 \bint f ds^{n-2}= \frac{1}{n} c_n^{-\frac{n-2}{n}}\Res \big(f \Delta_g^{-\frac{n}{2}+1}\big)=\frac{2}{n}c_n^{-\frac{n-2}{n}}\Res_{s=\frac{n}2-1} \Tr\big[ f\Delta_g^{-s}\big]. 
 \label{eq:EqEH.bint-residue-zeta} 
\end{equation}
As is well-known (see, e.g., \cite{GS:JGA96}), the residues of the zeta function $ \Tr\big[ f\Delta_g^{-s}\big]$ are related to the coefficients of the short time heat trace asymptotics, 
\begin{equation*}
 \Tr\big[ f e^{-t\Delta_g}\big] \sim_{t\rightarrow 0^+} \sum_{j\geq 0} t^{-\frac{n}2+j} \int_M f a_{2j}(\Delta_g)\nu(g). 
\end{equation*}
In particular, we have
\begin{equation}
 \Res_{s=\frac{n}2-1} \Tr\big[ f\Delta_g^{-s}\big]= \Gamma\big(\frac{n}{2}-1\big)^{-1}  \int_M f a_2(\Delta_g)\nu(g).
 \label{eq:EqEH.residue-zeta-a2}  
\end{equation}
The first few coefficients $a_{2j}(\Delta_g)$ are known (see~\cite{Gi:CRC95}). For $j=1$, we have
\begin{equation*}
 a_2(\Delta_g)=-\frac{1}{6} (4\pi)^{-\frac{n}2} \kappa(g). 
\end{equation*}
Combining this with~(\ref{eq:EqEH.bint-residue-zeta}) and~(\ref{eq:EqEH.residue-zeta-a2}) shows that $\bint f ds^{n-2}=-c_n'\int_M f\kappa(g)\nu(g)$ with $c_n'$ given by~(\ref{eq:quantized.cn'}). The proof is complete.  

%\clearpage
\section{$L^2$-Integration Formulas} \label{sec:L2-formulas} 
In this appendix we compare the integration formulas of this paper with the curved integration formula of~\cite{MPSZ:Preprint}. The latter is an $L^2$-version of the curved integration formula~(\ref{eq:curved.curved-formula}), in the sense that it holds for any element of $\cH_\theta=L^2(\T^n_\theta)$. The approach of~\cite{MPSZ:Preprint} is based on the trace formula for noncommutative tori of~\cite[Theorem~6.5]{MSZ:MA19} and a ``curved'' $\cL^{1,\infty}$-Cwikel estimate which is proved in~\cite{MPSZ:Preprint}. The approach to the trace formula in~\cite{MSZ:MA19} relies on a $C^*$-algebraic approach to the principal symbol~\cite{MSZ:MA19, SZ:JOT18}. In~\cite{MPSZ:Preprint} the ``curved'' $\cL^{1,\infty}$-Cwikel estimate is deduced from a flat $\cL^{1,\infty}$-Cwikel estimate via various operator theoretic considerations.

The point we would like to make in this appendix is that in order to get $L^2$-integration formulas we only need the flat $\cL^{1,\infty}$-Cwikel estimate of~\cite{MPSZ:Preprint}. We shall actually get an $L^2$-integration formulas for \psidos\ (see~Theorem~\ref{thm:L2.PsiDO-Cwikel} below). The main result of~\cite{MPSZ:Preprint} then follows from this as a mere special case. In particular, this allows us to bypass the $C^*$-algebraic considerations of~\cite{MSZ:MA19, SZ:JOT18} and the operator theoretic considerations of~\cite{MPSZ:Preprint} to get an even more general result. 

To wit note that if $x\in \cH_\theta$, then the left multiplication by $x$ maps $\cA_\theta$ to $\cH_\theta$, but in general this operator need not be bounded. In what follows we denote by $\|\cdot\|_2$ the norm of $\cH_\theta$. The flat  $\cL^{1,\infty}$-Cwikel estimate of~\cite{MPSZ:Preprint} can be stated as follows. 

\begin{proposition}[Flat  $\cL^{1,\infty}$-Cwickel Estimate~\cite{MPSZ:Preprint}] \label{app:L2.flat-Cwikel} 
 For every $x\in \cH_\theta$ the operator $x(1+\Delta)^{-\frac{n}2}$ is bounded and weak-trace class. Moreover, there is $C>0$ independent of $x$ such that
\begin{equation}
 \big\|x(1+\Delta)^{-\frac{n}{2}}\big\|_{1,\infty}\leq C \|x\|_{2}.
 \label{eq:L2.flat-Cwikel} 
\end{equation}
\end{proposition}

We observe we actually have a pseudodifferential version of the above Cwikel estimate. 

\begin{proposition}[Pseudodifferential $\cL^{1,\infty}$-Cwikel Estimate] 
 Let $P\in\Psi^{-n}(\T^n_\theta)$. For every $x\in \cH_\theta$ the operator $xP$ is bounded and weak trace class. Moreover, there is $C_P>0$ independent of $x$ such that
\begin{equation}
 \big\|xP\big\|_{1,\infty}\leq C_P \|x\|_{2}.
 \label{eq:L2.PsiDO-Cwikel} 
\end{equation}
\end{proposition}
\begin{proof}
  As $P$ has order~$-n$, the operator $(1+\Delta)^{\frac{n}{2}}P$ is a zeroth order \psido\ by Corollary~\ref{cor:PsiDOs.composition-classical}, and so it is bounded by Proposition~\ref{prop:PsiDOs.boundedness}. Let $x\in \cH_\theta$. The first part of Proposition~\ref{app:L2.flat-Cwikel} ensures us that $x(1+\Delta)^{-\frac{n}2}$ is bounded and weak-trace class, and so the operator $xP=x(1+\Delta)^{-\frac{n}2}\cdot (1+\Delta)^{\frac{n}{2}}P$ is in the ideal $\scL^{1,\infty}$. Moreover, we have 
\begin{equation*}
 \|xP\|_{1,\infty} \leq \| x(1+\Delta)^{-\frac{n}2} \cdot (1+\Delta)^{\frac{n}{2}}P\|_{1,\infty} \leq  \| x(1+\Delta)^{-\frac{n}2}\|_{1,\infty}  \|(1+\Delta)^{\frac{n}{2}}P\|. 
\end{equation*}
 Combining this with~(\ref{eq:L2.flat-Cwikel}) gives the result.
\end{proof}

\begin{remark}
 Applying the above result to $P=(1+\Delta_g)^{-\frac{n}2}$ allows us to recover the curved $\cL^{1,\infty}$-Cwikel estimate of~\cite{MPSZ:Preprint}. 
\end{remark}

We are now in a position to get an $L^2$-version of Corollary~\ref{cor:Trace-thm.super-integration-formula}. 

\begin{theorem}\label{thm:L2.PsiDO-Cwikel}
 Let $P\in\Psi^{-n}(\T^n_\theta)$. For every $x\in \cH_\theta$ the operator $xP$ is strongly measurable, and we have
\begin{equation}
 \bint x P= \frac{1}{n} \tau\big[xc_P],
 \label{eq:L2.PsiDO-integration}
\end{equation}
where $c_P$ is defined as in~(\ref{eq:NCR.cP}). 
\end{theorem}
\begin{proof}
In the same way as in the proof of Corollary~\ref{cor:Trace-thm.super-integration-formula}, in order to show that $xP$ is strongly measurable and satisfies the trace formula~(\ref{eq:L2.PsiDO-integration}), we only have to show that, for every positive trace $\varphi$ on $\scL_{1,\infty}$, we have
\begin{equation*}
 \varphi(xP)=\frac{1}{n}\tau\big[xc_P\big]. 
\end{equation*}
We know  by Corollary~\ref{cor:Trace-thm.super-integration-formula} that the equality holds for all $x\in A_\theta$. It follows from the continuity of $\varphi$ and the estimate~(\ref{eq:L2.PsiDO-Cwikel}) that the l.h.s.~is a continuous linear form on $\cH_\theta$. As $c_P\in \cA_\theta$ the r.h.s.~is continuous as well. Combining this with the density of $A_\theta$ in $\cH_\theta$ we then deduce that the equality holds for any $x\in \cH_\theta$. This completes the proof. 
\end{proof}

By using Theorem~\ref{thm:L2.PsiDO-Cwikel} and arguing along the same lines as that of the proof of Theorem~\ref{thm:Curved.integration-formula} we recover the main result of~\cite{MPSZ:Preprint} in the following form. 

\begin{corollary}[\cite{MPSZ:Preprint}] Let $g\in \GL_n(\cA_\theta^\R)$ be a Riemannian metric on $\T^n_\theta$. We have 
\begin{equation*}
 \bint x^\circ \Delta_g^{-\frac{n}2} =\bcn \tau\big[x\tilde{\nu}(g)\big] \qquad \forall x \in \cH_g^\circ. 
\end{equation*}
 \end{corollary}

\end{document}